\documentclass[12pt]{article}
\usepackage{amsthm,amssymb,amsmath}

\newcommand{\C}{\mathbb C}
\newcommand{\PP}{\mathbb P}
\newcommand{\Q}{\mathbb Q}
\newcommand{\R}{\mathbb R}
\newcommand{\Z}{\mathbb Z}

\newcommand{\fM}{\mathfrak M}

\newcommand{\sB}{\mathcal B}
\newcommand{\sD}{\mathcal D}
\newcommand{\sF}{\mathcal F}
\newcommand{\sH}{\mathcal H}
\newcommand{\sK}{\mathcal K}
\newcommand{\sM}{\mathcal M}
\newcommand{\sO}{\mathcal O}

\newcommand{\bnd}{{}^{\mathrm{b}}}
\newcommand{\Cr}{{}_{\mathrm{c}}} 
\newcommand{\hor}{^{\mathrm{h}}}
\newcommand{\lcm}{{}_{\mathrm{lcm}}}
\newcommand{\Lg}{^{\mathrm{log}}}
\newcommand{\linsys}[1]{\left|{#1}\right|} 
\newcommand{\mc}{{}_{\mathrm{mlcc}}}
\newcommand{\nor}{^{\mathrm{n}}}
\newcommand{\pt}{\mathrm{pt.}}
\newcommand{\red}{{}_{\mathrm{red}}}
\newcommand{\rest}[1]{{}_{{\textstyle{|}}#1}} 
\newcommand{\sdiv}{{}_{\mathrm{div}}}
\newcommand{\smod}{{}_{\mathrm{mod}}}
\newcommand{\st}{{}_{\mathrm{st}}} 
\newcommand{\tdlt}{{\rm tdlt}}
\newcommand{\tlc}{{\rm tlc}} 
\newcommand{\umm}{{}^{\mathrm{mm}}}
\newcommand{\umod}{^{\mathrm{mod}}}

\newcommand{\1}{^{-1}}

\DeclareMathOperator{\aBir}{Bir_\diamond}
\DeclareMathOperator{\alg}{\mathcal R}
\DeclareMathOperator{\Aut}{Aut}
\DeclareMathOperator{\Bir}{Bir}
\DeclareMathOperator{\bNS}{b-NS}
\DeclareMathOperator{\bPic}{b-Pic}
\DeclareMathOperator{\chr}{char}
\DeclareMathOperator{\Diff}{Diff}
\DeclareMathOperator{\dk}{\iota} 
\DeclareMathOperator{\Gal}{Gal}
\DeclareMathOperator{\GM}{Models}
\DeclareMathOperator{\mBir}{Bir_{\rm mp}}
\DeclareMathOperator{\id}{Id}
\DeclareMathOperator{\Ii}{i}
\DeclareMathOperator{\Lat}{Lattices}
\DeclareMathOperator{\mult}{mult}
\DeclareMathOperator{\NS}{NS}
\DeclareMathOperator{\Pic}{Pic}
\DeclareMathOperator{\prd}{p}
\DeclareMathOperator{\res}{res}
\DeclareMathOperator{\Spec}{Spec}
\DeclareMathOperator{\Supp}{Supp}

\theoremstyle{definition}
 \newtheorem{con}{Conjecture}
 \newtheorem{df}{Defenition}

\theoremstyle{plain}
 \newtheorem{cor}{Corollary}
 \newtheorem{lemma}{Lemma}
 \newtheorem{prop}{Proposition}
 \newtheorem{propdf}{Proposition-Definition}
 \newtheorem{thm}{Theorem}

\theoremstyle{remark}
 \newtheorem{ex}{Example}
 \newtheorem{exs}{Examples}

\title{Log Adjunction: effectiveness and positivity}

\date{May 26th 2006, Moscow/August 22nd 2013, Moscow}

\author{V.V. Shokurov
\thanks{Partially supported
by NSF grants DMS-0400832, DMS-0701465 and DMS-1001427.} }

\begin{document}

\maketitle

\section{Introduction}

This is a first instalment of much larger work
about relations between birational geometry and moduli
of triples.
The extraction of work is mainly related to
Theorem~\ref{weak_Kawamata}.
It is a weak version of Kawamata's Conjecture~\ref{Kawamata_conjecture}
and an important technical step toward
semiampleness of moduli part of adjunction.
To prove Theorem~\ref{weak_Kawamata}, we use
relative analogues of b-representations.
The proof here is rather complete except for
b-mobile property used in Corollary~\ref{adjunt_for_0contr}.
We assume also the LMMP and the semiampleness (abundance) conjecture.
For the former, it is sufficient \cite{BCHM}.
The latter is not crucial for b-representations, because
nonabundace gives empty representations.
This will be cleared up  in a final version of the preprint.

The preprint will be periodically renewed on
http://www.math.jhu.edu/$\sim$shokurov/adj.pdf.
A final version will appear again on arXive.

The author is grateful to
Florin Ambro, Valery Alexeev, J\'anos Koll\'ar for
sharing unpublished materials
and their valuable expertise in the area where
he is still an apprentice.

\section{Adjunction}

\begin{propdf}[Maximal log pair] 
Let $(X_\eta,B_{X_\eta})$ be a generic wlc pair
with a boundary $B_{X_\eta}$.
Then there exists a {\em maximal\/} complete wlc pair $(X_m/Z_m,B_m)$,
which is birationally equivalent to $(X_\eta,B_{X_\eta})$,
that is, there exists a flop
$$
(X_\eta,B_{X_\eta})\dashrightarrow (X_{\eta_m,m},B_{X_{\eta_m,m}}),
$$
where $\eta_m$ is a generic point of $Z_m$ and
the flop induces an isomorphism $\eta\cong\eta_m$.
The {\em maximal\/} property means an inequality
$\sB_m\umod\ge{\sB'}\umod$
for any complete wlc pair $(X'/Z',B')$,
which is birationally equivalent to $(X_\eta,B_{X_\eta})$.
For a maximal pair it is not necessary that $(X_m,B_m)$
is lc and $B_m$ is a boundary, but $(X_m,B_m)$ is
a log pair and $K_{X_m}+B_m$ is nef over $Z_m$.
However, always there exists a wlc maximal pair $(X_m/Z_m,B_m)$.

If $(X/Z,D)$ is an (irreducible) pair, which
is generically a wlc pair, then
its {\em maximal\/} pair is a maximal complete wlc pair of
$(X_\eta,D_{X_\eta})$, where $\eta$ is a generic point of $Z$,
in particular, $D_{X_\eta}$ is a boundary.
In this situation we denote a maximal moduli part of adjunction
by $\sD\umod$.
\end{propdf}

\begin{proof}
Immediate by the existence of a complete \tdlt\ family for $(X_\eta,B_{X_\eta})$ and
by Proposition~\ref{can_moduli_part}.
\end{proof}

\begin{exs}
(1) ($0$-mappings.)
Let $(X/Z,D)$ be a complete (irreducible) log pair such that
\begin{description}

\item{}
the generic fiber is a $0$-pair, possibly, not geometrically
irreducible, but $D$ is a boundary generically over $Z$;
and

\item{}
$K+D\equiv 0/Z$.
\end{description}
The complete property is global, that is,
$X$ and $Z$ are complete.
Then $(X/Z,D)$ is maximal itself \cite{PSh}.

In particular, if $f\colon X\to Z$ is a contraction, it is
a $0$-contraction and the maximal (upper) moduli part of
adjunction is $\R$-linear equivalent to
a pulling back of low moduli part of adjunction
(see Corollary~\ref{adjunt_for_0contr}):
$$
\sD\umod\sim_\R f^*\sD\smod.
$$

(2) (Maximality over curve; cf. a canonical moduli part
in Proposition~\ref{can_moduli_part}.)
Let $(X/C,B\Lg)$ be a complete \tlc\ pair such that
\begin{description}

\item{}
$C$ is a nonsingular complete curve,

\item{}
a generic fiber is wlc,
and

\item{}
$K+B\Lg$ is nef over $C$.
\end{description}
Then the pair is maximal.
The \tlc\ in this situation means
that $B\Lg=B+\sum D_i$ is a boundary such that
the vertical part of $B\Lg$ is a sum $\sum D_i$ of
reductions of fibers $D_i=(f^*p_i)\red,p_i\in C$,
the vertical sum includes all degenerations and
$(X,B\Lg)$ is wlc.
The inclusion of degenerations means that, if $p\in C\setminus\{p_i\}$,
then $(X,B\Lg+f^*p)$ is also wlc.
In particular, the fiber $f^*p$ is reduced.
The log structure of $(X/C,B\Lg)$ is given on $X$ by
the reduced horizontal divisors and reduction of vertical
degenerations $D_i$ of $f$, and
on $C$ by the (critical) points $p_i=f(D_i)\in C$.
A maximal (upper) moduli part $\sB\umm$ of adjunction for $(X/C,B\Lg)$ or for $(X/C,B)$
is stabilized over $X$ and is
$\overline{K+B\Lg-f^*(K_C+\sum p_i)}$.
Indeed, $B\sdiv\Lg=\sum p_i$.

Of course, we can add to $B\Lg$ some nondegenerate fibers $f^*p$ as above.
This does not change $\sB\umm$.
Moreover, if $D=B\Lg +f^*A$, where $A$ is any divisor on $C$, then
$(X/C,D)$ is also maximal and has the same moduli part as for $(X/C,B\Lg)$.
So,
$$
\sD\umm=\overline{D\umm}=\sB\umm=\overline{K+B\Lg-f^*(K_C+\sum p_i)}.
$$
Note that $K+B\Lg-f^*(K_C+\sum p_i)$ is
a divisor of a (log) canonical $\R$-sheaf
$\omega_{X/C}^1[B]$, an adjoint log sheaf,
where $B=B\Lg{}\hor=D\hor$ is the horizontal part of $B\Lg,D$.
If the fibers of $f$ are reduced, then $\omega_{X/C}=\omega_{X/C}^1$.

If $X$ is also a complete nonsingular curve,
then $X\to C$ is a finite morphism of curves, $B=0$, and $D,B\Lg=\sum q_{j,i}$ are vertical.
In this situation, $\sD\umm=D\umm\sim 0$, $B\Lg=\sum q_{j,i},q_{j,i}\in X$,
with $\sum_jq_{j,i}=D_i$, and
the above equation
$$
K+\sum q_{j,i}-f^*(K_C+\sum p_i)\sim 0
$$
is the Hurwitz formula.
The points $p_i\in C$ should include all critical ones.
\end{exs}

\begin{cor}\label{adjunt_for_0contr}
Let $(X/Z,D)$ be a complete (irreducible) log pair such that
\begin{description}

\item{}
$f\colon X\to Z$ is a contraction,

\item{}
the generic fiber is a $0$-pair, and

\item{}
$K+D\equiv 0/Z$.
\end{description}
Then $(X/Z,D)$ is maximal and
$$
\sD\umm=\sD\umod\sim_\R f^*\sD\smod.
$$
Moreover,
there exists an effective low moduli part $\sD\smod$
such that
\begin{description}

\item{}
the moduli parts $\sD\umod=f^*\sD\smod,D\umod=(\sD\umod)_X$,
and
$\sD\smod,D\smod=(\sD\smod)_Z$ are also effective  and flop invariant:
for every flop $g\in\Bir(X\to Z/k,D)$,
$$
g^*\sD\umod=\sD\umod,g^*D\umod=D\umod \text{ and }
g_Z^*\sD\smod=\sD\smod,g_Z^*D\smod=D\smod,
$$
where $g_Z\colon Z\dashrightarrow Z$ is
a birational automorphism induced by $g$;

\item{}
if $(X,D)$ is lc, klt, then
$(Z,D_Z)$ is lc, klt respectively,
where $D_Z=D\sdiv+D\smod$;

\item{}
if $D$ is a effective, then
the divisorial part $D\sdiv$,
the above moduli part $D\smod$ on $Z$ and
$D_Z$ are effective $\R$-divisors;

\item{}
if $D=B$ is a boundary, then
the divisorial part $B\sdiv=D\sdiv$,
the above moduli part $B\smod=D\smod$ on $Z$ and
$B_Z=B\sdiv+B\smod$ are boundaries;
and

\item{}
if $(X,Z)$ is a wlk (klt) pair, then
the pair $(Z,B_Z)$ is wlc (klt respectively).
\end{description}
\end{cor}

\begin{prop}[Canonical moduli part]\label{can_moduli_part}
Let $(X/Z,B)$ be a \tlc\
wlc (irreducible) family with a horizontal boundary $B$.
Then the pair is maximal,
its maximal moduli part is stabilized over $X$, and
any divisor $M$ of $\R$-sheaf $\omega_{X/Z}^1[B]$ is
a divisor of the upper maximal moduli part of adjunction.
In particular, it is the maximal moduli part for $(X_\eta,B_\eta)$,
where $\eta$ is a generic point of $Z$.
More precisely,
$$
B\umm=B\umod\sim M=K_{X/Z}\Lg+B
$$
and
$$
\sB_\eta\umm=
\sB_\eta\umod=
\sB\umm=\sB\umod=\overline{B\umm}=\overline{B\umod}
\sim\overline{M}=\sK_{X/Z}\Lg+\overline{B}.
$$
\end{prop}

\begin{df}[Canonical adjunction]
A {\em canonical (upper) maximal moduli part of
adjunction\/} is the $\R$-sheaf $\omega_{X/Z}^1[B]$ of
Proposition~\ref{can_moduli_part}.
It is a b-sheaf on $(X_\eta,B_{X_\eta})$.
Such a b-sheaf is unique on $X_\eta$ and
some times we denote it by $\sM$.
By $\sM$ we denote also a b-divisor of the last b-sheaf.
This divisor is defined up to linear equivalence and
$$
\sM=\overline{M},
$$
where $M$ is a divisor in Proposition~\ref{can_moduli_part}.

Respectively, a plane {\em moduli part of adjunction\/} $\sM$ is
either the $\R$-sheaf $\omega_{X/Z}^1[B]$ up to an $\R$-isomorphism, or
the $\R$-divisor $\sM$ up to $\R$-linear equivalence.
To define a low moduli part of adjunction
we need such a flexibility.
\end{df}

\section{Mapping $\prd$}

\begin{df}[Equivalent $0$-pairs]
The {\em equivalence\/} of connected $0$-pairs $(X,B)$ is
the minimal equivalence such that
\begin{description}

\item{\rm (1)}
component adjunction gives an equivalent $0$-pair,
that is,
any component $(X_i,B_i)$ of
a normalization $(X,B)\nor=\coprod (X_i,B_i)$
is a $0$-pair equivalent to a $0$-pair $(X,B)$ itself;

\item{\rm (2)}
any flopped $0$-pairs are equivalent, that is,
if $(X,B)\dashrightarrow (X',B_{X'})$ is
a flop of pairs with boundaries and $(X,B)$ is a $0$-pair,
then $(X',B_{X'})$ is a $0$-pair equivalent to $(X,B)$;

\item{\rm (3)}
divisorial adjunction gives an equivalent $0$-pair,
that is,
if $D\subset (X,B)$ is a divisorial lc center
of a $0$-pair $(X,B)$,
then the adjoint pair $(D,B_D)$ is
a $0$-pair equivalent to $(X,B)$; and

\item{\rm (4)}
field base change gives an equivalent $0$-pair,
that is,
if $(X,B)$ is a $0$-pair over a field $K/k$ and
$F/k$ is a field extension,
then any connected component of pair $(X,B)\otimes_kF$
is a $0$-pair over $F$ equivalent to $(X,B)$ over $K$.
\end{description}
\end{df}

\begin{exs}

(1)
(Log curves.)
Let $(C,B_C)$ be a $1$-dimensional $0$-pair,
that is, $C$ is a connected complete nodal curve and
$B_C$ is a boundary such that $K_C+B_C\sim_\R 0$.
Such a pair $(C,B_C)$ is equivalent to a $0$-dimensional
$0$-pair $(\pt,0)$ if $(C,B_C)$ is not klt,
equivalently, it has a node or
a nonsingular point $p\in C$ with $\mult_pB_C=1$.

Two $1$-dimensional klt $0$-pairs $(C,B_C)$ and
$(C',B_{C'})$ are equivalent if and only if
they are log isomorphic.

(5)
(Toric pairs.)
Any complete toric variety $X$ is
naturally a $0$-pair $(X,D)$, where
$D$ is its total invariant divisor.
All those pairs are equivalent and
they equivalent to $(\pt,0)$ of Example (1).

(6) (Koll\'ar's sources \cite{Kol11}.)
Let $(X,\Delta)$ be a log pair,
$Z$ be its lc center,  and
$\Delta\ge 0$ near a generic point of $Z$.
Then one can associate to $Z$ a class of $0$-contractions
$(S/\widetilde{Z}_S,\Delta_S)$, a relative source.
The class of pairs $(S,\Delta_S)$ up to flops is denoted by $\text{Scr}(Z,X,\Delta)$
and is called the source of $Z$ in $(X,\Delta)$.
An interest to the class is related to the divisorial part of adjunction
(see for more details in \cite[Theorem~1]{Kol11}).
On the other hand, the moduli part of adjunction is
related to the equivalence class of generic pairs $(S_\eta,B_{S_\eta})$
of relative sources,
where $\eta$ is a generic point of $\widetilde{Z}_S$ (see for Corollary~\ref{adjunt_for_0contr}).
The mapping of sources into equivalence classes is not
injective, except for, the case with $\eta=\widetilde{Z}_S=\pt$ and
$Z=\pt$.
The divisorial and  moduli part in this case are $0$ and $\sim_\R 0$ respectively.

(7) (Characteristic.)
Let $k$ be a prime field and $F/k$ is a field extension.
Then by definition
the $F$-point $\pt_F=(\Spec F,0)$ is
equivalent to $\pt=(\Spec k,0)$:
$\pt_F=\pt\otimes_kF$.
So, the equivalence class of $\pt$ is the class of
$\pt_F$ for field extensions $F/k$.
This class of fields is uniquely determined by
$\chr k$.
\end{exs}

\begin{prop} \label{klt0pair}
Every equivalence class of $0$-pairs has
an (irreducible) projective klt representative $(X,B)$.
Two klt $0$-pairs over $k$ are equivalent
if and only if they are related by
a generalized flop.
\end{prop}

\begin{lemma} \label{mincents}
Let $(X,B),(X',B')$ be
connected \tdlt\ $0$-pairs,
$(V,B_V),(V',B_{V'}')$ be irreducible adjoint lc centres $V,V'$
of those pairs respectively such that
there exists a flop
$$
(V,B_V)\dashrightarrow(V',B_{V'}'),
$$
and
$(W,B_W),(W',B_{W'}')$ be adjoint pairs of
minimal lc centers $W,W'$ of $(X,B),(X',B')$ respectively.
Then $(W,B_W),(W',B_{W'}')$ are klt $0$-pairs and
are related by a (generalized) flop.
In particular, the conclusion holds
for adjoint pairs $(W,B_W),(W',B_{W'})$ of
any two minimal lc centers $W,W'$ of $(X,B)$.
\end{lemma}

\begin{proof}
Immediate by Theorem~\ref{can_0-perest} and Lemma~\ref{chain_of_flop}.
\end{proof}

\section{Toroidal geometry}

\begin{prop}\label{torus_invar-pt}
Let $X\subseteq \PP(W^v)$ be a nondegerate projective variety, and
$g$ be a semisimple operator on $W$.
Suppose that $X$ is invariant under
the dual (contragarient) action of $g^v$.
Then there exists an integral number $n\not=0$ and a point $x\in X$ such that
\begin{description}

\item{\rm (1)}
$(g^v)^n$ invariant: $(g^v)^nx=x$,
and, moreover,

\item{\rm (2)}
if $g$ is not torsion, then
the point $x$ has a nonzero coordinate $w(x)$,
where $w$ is a coordinate (linear) function,
an eigenfunction under $g$
with an eigenvalue, which is not a root of unity.
\end{description}
\end{prop}

\begin{proof}
Take a basis $w_1,\dots,w_d,d=\dim W$, with
eigenvectors $w_i$ for $g$ such that
the eigenvalues $e_1,\dots,e_l$ of $w_1,\dots,w_l,0\le l\le d$,
respectively  are nonroots of unity and
the eigenvalues $e_{l+1},\dots,e_d$ of $w_{l+1},\dots,w_d$
respectively are roots of unity.
(Actually, for the dual action $g^v\colon w_i^v\mapsto e_iw_i^v$,
because the representation is commutative.)

If $l=0$, take any $x\in X$ and a uniform torsion $n$ of all roots
of unity $e_i$.

If $l\ge 1$, take a sufficiently general point $x\in X$, that is,
all homogeneous coordinates $x_i=w_i(x)\not=0$.
Then the Zariski closure of the orbit $(g^v)^m(x),m\in\Z$,
in $\PP(W^v)$ is the closure $\overline{Y}$ of a subtoric orbit $Y$ with
respect to the coordinate system $w_i^v$,
that is, the torus action by diagonal matrices
in this basis.
By construction $Y\subseteq\overline{Y}\subseteq X$.
The group generated by $g^v$ is dense in the subtorus $T$
in the Zariski topology.
Let $T_1$ be the connected component of the unity $1$.
Then $n=\# T/T_1$, a torsion of the Abelian quotient $T/T_1$.

The subvariety $\overline{Y}$ is toric with respect to $T_1$, and
has a $T_1$-invariant point $y$
with some homogeneous coordinate $w_i(y)\not=0,1\le i\le l$,
where $T_1=T_1^n=T^n$.
So it is $(g^v)^n$-invariant (1).
(We identify the point $y=(y_1:\dots:y_d)\in X\subseteq \PP(W^v)$
with a line $k(y_1,\dots,y_d)$ in $W^v$.)

The subtorus $T_1$ has the following parameterization:
$$
k^{*d}\twoheadrightarrow T_1, (t_1,\dots,t_d)\mapsto
(\prod_j t_j^{a_{j,i}}), a_{j,i}\in \Z,
i,j=1,\dots,d,
$$
with the weighted action:
$$(w_i^v)\mapsto (\prod_j t_j^{a_{j,i}}w_i^v).
$$
(We can suppose that this is an action of the whole
torus $k^{*d}$.
Actually, under our assumptions, the action is smaller:
all $a_{j,i}=0$ for $i\ge l+1$.)
The weight of vector $w_i^v$ or of coordinate $x_i$
is the vector $(a_{1,i},\dots,a_{d,i})$.
The wights are ordered lexicographically:
$(a_1,\dots,a_d)\ge (b_1,\dots,b_d)$ if
$a_1>b_1$ or $a_1=b_1,a_2>b_2$, etc.

We can assume that the action is {\em positive\/}.
This implies (2).
The positivity means that the maximal weight vector
$(a_{j,i})$ is {\em positive\/}: all $a_{j,i}\ge 0$ and some $a_{j,i}>0$.
Under our assumptions, $i\le l$.
If the action is not positive, then
changing the action of $t_j$ for some $j$ by the inverse one $t_j\1$
(equivalently, change $a_{j,i}$ on $-a_{j,i}$),
any action can be converted into a positive one.
The positivity of action allows to get
an invariant point with $w_i(y)\not=0,1\le i\le l$,
if $w_i^v$ is maximal.

More precisely, a $T_1$-invariant vector $y\in \overline{Y}$
can be constructed as follows.
The vector $y=(y_i)$ has the coordinates $y_i=x_i$,
if $w_i^v$ has the maximal weight and
$y_i=0$ for the other coordinates.
Then $y\in \overline{Y}$ (the closure of orbit $T_1x=T_1(x_i)$ of $x$) and
(1) $T_1$-invariant  with (2) $w_i(y)=x_i\not=0$ for any maximal $w_i^v$.
\end{proof}

\section{Isomorphisms and flops}

\begin{lemma} \label{etal-action}
Let $(X,B)$ be a wlc klt pair, and
$\sD$ be a b-polarization on $X$.
The natural mapping
$$
\alpha_\sD\colon \Bir_0(X,B)=\Aut_0(X,B)\to \bPic_\sD X, a\mapsto \text{class }\sO_X(a^*\sD),
$$
is an isogeny of Abelian varieties on the image.
Thus $\Aut_0(X,B)$ is an Abelian variety.

The fiber $G_\sD=\alpha_\sD\1(\alpha_\sD \sD)\subseteq \Aut_0(X,B)$
is a subgroup and coincide with
the kernel of natural homomorphism
$$
\gamma_\sD\colon \Aut_0(X,B)\to \Pic_0 X, a\mapsto \text{class}\sO_X(a^*\sD-\sD).
$$
More precisely,
$$
G_\sD=
\ker\gamma_\sD=\Aut(X,B,\linsys{\sD})\cap
\Aut_0(X,B)\subseteq \Aut_0(X,B)
$$
is a finite Abelian group and depend only on
the algebraic (not numerical) equivalence class:
$$
\sD\approx \sD'\Rightarrow G_\sD=G_{\sD'},A_\sD=A_{\sD'}.
$$

For a projective $0$-pair $(X,B)$,
$\alpha_\sD$ is an isogeny onto.
In general (proper) case $\Aut_0(X,B)$ should be
replaced by $\Bir_0(X,B)$.
\end{lemma}

\begin{thm}\label{aut_tame}
Let $(X,B,H)$ be a klt wlc triple with a polarization $H$,
an ample sheaf or an ample divisor up to algebraic or
up to numerical equivalence.
Then the group
$$
\Bir(X,B,H)=\Aut(X,B,H)
$$
is tame.
More precisely, the group is algebraic of finite type and
complete (almost Abelian).
\end{thm}

\begin{df}
Let $(X/T,B)$ be a connected family.
The family is {\em moduli part trivial}, for short,
{\em mp-trivial\/}, if its upper moduli part of adjunction
$\sM$ behaves on $X$ as on a trivial fibration:
$\Ii(X,\sM)=\Ii(X/T,\sM)$ or, equivalently,
there are rather general horizontal curves $C\subseteq X$
over $T$ such that $(\sM.C)=0$.

If $(X/T,B)$ is a family of $0$-pairs, then the mp-trivial
property means that $\sM\sim_\R0$, as a b-divisor.

Respectively, the family {\em isotrivial}, if
its rather general fibers are log isomorphic.
\end{df}

\begin{ex}\label{mp_triv_nonisotri}
Let $(X/C,S+B)$ be a $\PP^1$-fibration over
a nonsingular curve $C$ with a section $S$ and
a boundary $B=\sum b_iD_i$, where
prime divisors $D_i$ are horizontal.
Suppose also that $(X/C,S+B)$ is \tlc, and
$(X_\eta,B_{X_\eta})$ is a $0$-log pair.
Then $(X/C,S+B)$ is a maximal family of $0$-pairs and
its moduli part is trivial: $\sM\sim_\R 0$.
So, the family is mp-trivial.
However, for rather general divisors $D_i$, it not isotrivial.
\end{ex}

\paragraph{Directed generic flop.}
Let $g\colon (X/T,B)\dashrightarrow (Y/S,B_Y)$ be
a generic flop, that is, $g\in \Bir(X,B;Y,B_Y)$
is compatible with the relative structures.
The latter means that $g$ induces
a birational transformation $g_T\colon T\dashrightarrow S$
such that the flop if fiberwise with respect to $g_T$.

A {\em directed flop with respect to a b-polarization\/} $\sD$
and its decomposition $g=g\rest{Y'}c$.
It determines a b-polarization $\sD=g^*\sH$ on $X$,
where $\sH=\overline{H}$ is a b-divisor of polarization on $Y/S$.
However, $g_T,g_\eta$ and $g_\theta$ are not uniquely
determined by $\sD$.
The decomposition also depends on a polarization $\sH$.

\begin{df}
A generic flop $g\in\Bir(X\to T/k,B)$
is an {\em mp-autoflop\/}, if it transforms any fiber
$(X_t,B_{X_t})$ into a fiber $(X_{t'},B_{X_{t'}}),t'=g_T t$,
in a connected mp-trivial subfamily with $(X_t,B_{X_t})$.

Respectively, the flop is {\em almost autoflop\/}, if it
transforms rather general fibers within isotrivial connected families:
for rather general $t$, the fibers $(X_t,B_{X_t}),(X_{t'},B_{X_{t'}})$
are in a connected isotrivial subfamily.
\end{df}

The mp-autoflops form a normal subgroup
$\mBir(X\to T/k,B)\subseteq \Bir(X\to T/k,B)$.
Respectively, the almost autoflops. form a normal subgroup
$\aBir(X\to T/k,B)\subseteq \Bir(X\to T/k,B)$.

\begin{lemma}\label{mp_it}
$\aBir(X\to T/k,B)\subseteq\mBir(X\to T/k,B)$,
if $\sM$ is semiample on $X$, and,
$=$ holds if the family $(X/T,B)$ is a connected, generically klt, e.g.,
the family is a connected, generically klt, wlc maximal family.
\end{lemma}

\begin{proof}
Indeed, the isotrivial property implies that the upper moduli part
exists and is mp-trivial on that family.
Hence adjunction implies the inclusion.
The converse does not hold in general (Example~\ref{mp_triv_nonisotri}).

Now suppose that the upper moduli part $\sM$ exists, stabilized and semiample
over $X$, e.g., this holds, if the family $(X/T,B)$ is wlc maximal.
(After a perturbation of $B$, one can suppose that $B,\sM$ is $\Q$-divisors.)
For a rather divisible natural number $m$, the linear system $\linsys{m\sM}$
gives a contraction with mp-trivial fibers.
In this situation, $\mBir(X\to T/k,B)$ is exactly the kernel of b-representation:
$$
\Bir(X\to T/k,B)\to \Aut H^0(X,m\sM), g \mapsto g^*.
$$
The kernel can be determined on finitely many rather general fibers of
a morphism given by $\linsys{m\sM}$.
The fibers are mp-trivial by definition.
Rather general fibers are klt and isotrivial (Viehweg-Ambro),
if the family is generically klt \cite[Theorem~6.1]{Am}.
\end{proof}

\begin{cor}\label{it_generic_Kawamata}
Let $(X_\eta,B_{X_\eta})$ be a generic projective wlc klt
pair.
Then there are only finitely many generic
log flops of $(X_\eta\to \eta/k,B_{X_\eta})$ modulo
almost autoflops, that is,
the group
$$
\Bir(X_\eta\to\eta/k,B_{X_\eta})/\aBir(X_\eta\to\eta/k,B_{X_\eta})
$$
is finite.
\end{cor}

\begin{proof}
Immediate by Lemma~\ref{mp_it} and Corollary~\ref{generic_Kawamata}.
\end{proof}

\begin{ex}\label{Mordell-Weil}[Mordell-Weil group.]
Let $S$ be a surface with a nonisotrivial pencil $f$
of genus $g\ge 1$ curves.
We consider the pencil as a rational contraction $f\colon S\dashrightarrow C$
onto a curve $C$.
Then, by Corollary~\ref{it_generic_Kawamata}, the group
$\Aut(S/C)$
has finite index in the group $\Aut(S\dashrightarrow C/k)$
fixing the pencil.
Moreover, we can replace $\Aut(S/C)$ by $\Bir(S/C)$.
Indeed, we can replace $S/C$ by a genus $g$ fibration,
a minimal model over $C$ with genus $g$ fibers.
By our assumption, generic isotrivial subfamilies of $S/C$ are
$0$-dimensional (points) and thus
$$
\aBir(S\to C/k)=\Aut(S/C).
$$
For $g\ge 2$, groups $\Aut(S/C)$ and $\Aut(S\dashrightarrow C/k)$
are finite.
For $g=1$, $\Aut(S/C)$ is the Mordel-Weil group and
can be infinie.
\end{ex}

A first step to Corollary~\ref{it_generic_Kawamata} is as follows.

\begin{lemma}\label{Burnside_3}
Let $(X_\eta,B_{X_\eta})$ be
a generic projective klt $0$-pair, geometrically irreducible, and
$(X_\eta,B_{X_\eta})\to (\theta,\sM_\theta)$ be
an isotrivial contraction with a canonical polarization $\sM_\theta$,
the $\R$-direct image of a canonical sheaf moduli part of adjunction for $(X_\eta,B_{X_\eta})$.
Then there exists a extension $l\subset k$ of finite type of
the prime subfield such that the natural homomorphism
$$
\Bir(X_\eta\to \eta/k,B_{X_\eta})/\aBir(X_\eta\to\eta/k,B_{X_\eta})
\hookrightarrow\Aut(\theta/\overline{l},\sM_\theta), g\mapsto g_\theta,
$$
is injective and,
for every generic flop $g$,
there exists a finite subextension $l_g/l$ in $k$
of uniformly bounded degree such that
$g_\theta$ is defined over $l_g$.
\end{lemma}

\begin{proof} There exists a required $l$ over which $(X_\eta\to\eta,B_{X_\eta}),
(X_\eta,B_{X_\eta})\to (\theta,\sM_\theta)$
and polarizations $\sH_\eta,\sM_\theta$ are defined.
Suppose also that
$\bPic\bnd X_\eta$ (the Picard group of bounded b-divisors
in the sense of resolution) is defined over the same $l$.
The bound $\Pic\bnd$ means that we consider Cartier $b$-divisors which
are stabilized over some partial (geometric) resolution over $\eta$.
After a finite extension it can be given over $l$.
It is sufficient to consider any minimal model (i.e., geometrically $\Q$-Cartier)
which is defined for klt pairs or to take any log resolution.
Each generic flop $g\in \Bir(X_\eta\to \eta/k,B_{X_\eta})$
can be given by a geometric b-divisor $\sD\in \bPic\bnd X_\eta$.
Actually, it is sufficient such a divisor $\sD$ modulo algebraic
equivalence.
(The algebraic equivalence of b-divisors is the same as for
usual one, that is, modulo $\Pic_0 X_\eta$ considered as
a group scheme of divisors, fiberwise in the connected component of $0$.
We use here the rationality of klt singularities.)
Indeed, if $g\colon X_\eta\dashrightarrow X_\eta/k$ such a flop
then $\sD=g^* \sH_\eta$.
Since polarization is defined modulo the numerical equivalence,
we can take $\sD$ modulo algebraic equivalence as well.
After an extension of $l$ it has a representative in each
geometrically algebraic equivalence class, that is,
there exists a section $\sD\in\bPic\bnd X_\eta$ in
each those class.
(We treat $\bPic X_\eta$ as a scheme over $\eta$ and
b-divisors $\sD$ as its sections over $\eta$.)
This follows from the finite generatedness of
geometric divisors modulo algebraic equivalence
(the Neron-Severi group).

Next, we verify that each generic flop $g$ can be defined over
a finite extension $l_g/l$
modulo almost flops over $k$.
Taking a representative $\sD$ one can construct a flopped
variety $(X_\eta',B_{X_\eta'},\sH')$ with a canonical flop over $l$
$$
c_\sD=c\colon (X_\eta,B_{X_\eta},\sD)\dashrightarrow
(X_\eta',B_{X_\eta'},\sH'),c^*\sH'=\sD.
$$
It is over $\eta$, that is, identical on $\eta$:
$c_{\eta,\sD}=\id_\eta$.
The autoflop $g=g\rest{X_\eta'}c$ is given by a composition with
a log isomorphism (also a flop) of generic triples
$g\rest{X_\eta'}\colon (X_\eta,B_{X_\eta},\sH_\eta)\leftarrow(X_\eta',B_{X_\eta'},\sH')$,
which induces an automorphism $g_\eta$ of $\eta/k$.
In general, it does not preserve polarization and
is not defined over $l$.
However, $\sH'\approx
g\rest{X_\eta'}^*\sH_\eta/\eta$ and
Lemma~\ref{etal-action} implies fiberwise linear equivalence:
$$
\sH'_t\approx(g\rest{X_\eta'}^*\sH_\eta)_t
\Rightarrow \sH'_t\sim(g\rest{X_\eta'}^*\sH_\eta)_t,
$$
where $\sim$ up to isomorphism, that is,
there exists an autoflop
$h\colon (X_t',B_{X_t}')\dashrightarrow(X_t',B_{X_t}')$ with
$\sH'_t\sim h^*((g\rest{X_\eta'}^*\sH_\eta)_t)
=h^*g_t'^*(\sH_\eta)_{g_\eta t}$.
So, the triples $(X_t',B_{X_t'},\sH_t'),
(X_{g_\eta t},B_{X_{g_\eta t}},(\sH_\eta)_{g_\eta t})$ are equivalent
with polarizations up to $\sim$, where
the automorphism $g_\eta\colon \eta\to \eta$ is
induced by $g\rest{X_\eta'}$ or by $g$.
By the lemma $\sim$ is equal to $\approx$
up to isomorphism.

To construct required flops over $l_g$ we
use modili $\fM$ of triples for special fibers $(X_t,B_t,\sH_t)$
with polarization up to linear equivalence.
Such moduli exist.
By construction we have a unique morphism
$\mu=\mu'\colon\eta,\eta\to\fM$ corresponding to generic families
$(X_\eta,B_{X_\eta},\sH_\eta),(X_\eta',B_{X_\eta'},\sH')$.
Indeed, any generic flop preserves isotrivial families of wlc klt pairs, and
in our situation
preserves fiberwise the polarization up to $\sim$ as was explained above.
After an extension of $l$ we can suppose that $\fM$
is also defined over $l$.
The morphism $\mu$ is defined over $l$ too.
By construction, $\mu$ is defined over an algebraic closure
$\overline{l}$.
Let $\sigma\in \Gal(\overline{l}/l)$ be a Galois automorphism
then $\mu^\sigma=\mu$ and it is defined over $l$.
Indeed, $\mu^\sigma\colon \eta^\sigma=\eta\to\fM^\sigma=\fM$
is also universal because an isomorphism of triples
preserving polarization gives (conjugation) isomorphism of those triples:
$\sigma\colon (X_t,B_t,\sH_t)\cong(X_{t^\sigma}^\sigma,B_{t^\sigma}^\sigma,\sH_{t^\sigma}^\sigma)=
(X_{t^\sigma},B_{t^\sigma},\sH_{t^\sigma})$.
Then we use Hilbert 90.

By definition of generic flops there is a (birational) automorphism
$g_\eta\colon\eta \to \eta$.
Indeed, it is induced by the isomorphisms of families $g\rest{X_\eta'}$.
It is unique as $g\rest{X_\eta'}$ for the flop $g$ but is not unique itself.
However, for any other isomorphism as above
$(X_\eta,B_{X_\eta},\sH_\eta)\leftarrow (X_\eta',B_{X_\eta'},\sH')$,
an induced isomorphism $g_\eta'\colon\eta \to \eta$
is compatible with $g_\eta$ over $\fM$:
$\mu g_\eta'=\mu g_\eta$.
Equivalently, $g_\eta'g_\eta\1\in \Aut(\eta/\fM)$, that is,
preserving triples with polarization up to the linear equivalence.
By construction up to algebraic equivalence and
by Lemma~\ref{etal-action} up to linear one.
Thus $g_\eta=g_\eta'$ up to an automorphism of fibers of $\eta/\fM$.
The isomorphism $g_\eta$ is defined over $k$ and
in general the minimal field of definition for all $g_\eta$
can have algebraic elements of unbounded degree and
even infinite transcendent degree over $l$.
The main finite indeterminacy is the same as for almost auto flops.
To remove this, we push $g_\eta$ to
an automorphism $g_\theta\colon \theta\to \theta=\theta'$
where $\eta\to\theta=\eta\to \theta=\theta'$ is the universal morphisms
with maximal connected fibers over $\fM$ and $\theta=\theta'\to \fM$
is finite, that is, $\mu=\mu'$ can be universally and equally
decomposed $\eta\to\theta\to\fM=\eta\to\theta'\to\fM$.
(This decomposition can be obtain from a Stein one after
a completion of fibers over $\fM$, that is,
a completion of isotrivial families.)
The isotrivial property is preserved for $g_\eta$
because it induces a log isomorphism of fibers or,
equivalently, a flop preserves isotrivial klt families.
This holds fiberwise even for triples with polarization up to $\sim$.
Thus this is a single canonical decomposition $\eta\to \theta\to\fM$.
It is defined over the same field $l$ (after a finite extension)
independent of $g$.
By construction each $g$ preserves the canonical moduli part
of adjunction: $\omega_{X_\eta}^m[mB_{X_\eta}]=g^*\omega_{X_\eta}^m[mB_{X_\eta}]=
\omega_{X_\eta'}^m[mB_{X_\eta'}]$ (the last identification by
$g\rest{X_\eta'}^*\colon X_\eta\cong X_\eta'$).
This action commutes with the direct image on $\theta$:
$\sM_\theta^m=c_{\eta,*}\omega_{X_\eta}^m[mB_{X_\eta}]=
c_{\eta,*}g^*\omega_{X_\eta}^m[mB_{X_\eta}]=g_\theta^*c_{\eta,*}\omega_{X_\eta}^m[mB_{X_\eta}]=
g_\theta^*\sM_\theta^m$.
This concludes a construction of a required injection:
$$
\Bir(X_\eta\to \eta/k,B_{X_\eta})/\aBir(X_\eta\to\eta/k,B_{X_\eta})
\hookrightarrow\Aut(\theta/k,\sM_\theta),g\mapsto g_\theta.
$$
The kernel of map $g\mapsto g_\theta$ is $\aBir(X_\eta\to\eta/k,B_{X_\eta})$
by definition.

Finally, we verify that the image $g_\theta$ belongs
to $\Aut(\theta/l_g,\sM_\theta)$, where $l_g/l$ is a finite extension.
Actually, it depends on a choice of $\sD$ but it is unique up to
$\approx$ for a given $g$: $\sD\approx g^*\sH_\eta/\eta$.
Note that there are finitely many
conjugated automorphisms $g_\theta^\sigma$ of $g_\eta$ over $\overline{l}$.
More precisely,
$a_\theta=g_\theta\1g_\theta^\sigma\in \Aut(\theta/\fM)$
over $\overline{l}$.
Indeed, the conjugated flop $g'=g^\sigma$ is given by the same polarization
as above.
The divisors $\sH$, $\sD$ and $\sH'$ are defined over $l$.
(However, $g_\eta$ and $g_\theta$ are not uniquely
determined by these data.)
By construction
an automorphism $a=g_\eta^\sigma g_\eta\1$ induces the generic flop $g^\sigma g\1$
preserving fiberwise the triples of $(X_\eta,B_{X_\eta},\sH)$ over $\fM$, that is,
in $\Aut((X_\eta,B_{X_\eta},\sH)/\fM)$ over $\overline{l}$.
The push down induces $a_\theta\in \Aut(\theta/\fM)$ over $\overline{l}$.
The last group if finite of order $\le (\deg \theta/\fM)!$.
Thus any $g_\theta$ is defined over a uniformly bounded
extension $l_g$ of above $l$ of the degree $\le (\deg \theta/\fM)!$.
Each $g_\theta$ can be defined over an extension $l_g/l$ with
injective group action $\Gal(l_g/l)$ on the permutation group
of all $g_\theta^\sigma$.
\end{proof}

\section{Algebra and calculus of relative differentials}

\paragraph{Properties of the norm.}

(1) For every $\omega\in H^0(X,\omega_{X/T}^m[mB])$,
$\Vert g^*\omega\Vert_t=\Vert \omega\Vert_{g_Tt}$
where $g$ is a generalized flop of $(X/T,B)$
transforming $X_t\dashrightarrow X_{g_Tt}$.
So, $\Vert g^*\omega\Vert=\Vert\omega\Vert$.

(2) Let $\omega\in H^0(X,\omega_{X/T}^m[mB])$ be an
eigenvector for the induced linear operator $g^*$ with
an eigenvalue $\lambda\in \C$, that is, $g^*\omega=\lambda\omega$.
Then $\Vert g^*\omega\Vert=\vert\lambda\vert^2\Vert\omega\Vert$.

(3)
For every $\omega\in H^0(X,\omega_{X/T}^m[mB])$,
the function $\Vert\omega\Vert_t$ is continuous
in the complex (classical) topology on $T$
including the value $+\infty$.

\begin{lemma}\label{sum_resid}
Let $f\colon(X,B+D_1+D_2)\dashrightarrow T$ be
a rational conic bundle with two sections and
a vertical (sub)boundary $B$.
Then, for any rational $m$-differential $\omega$, that is
regular on the generic fiber,
$$
c^*(\omega\rest{D_2})=(-1)^m\omega\rest{D_1},
$$
where $c\colon (D_1,B_{D_1})\dashrightarrow (D_2,B_{D_2})$ is
a birational transformation given by the conic bundle structure.

Moreover, if $f$ is a $0$-contraction,
then $c$ is a flop and $c^*$ preserves
(as canonical isomorphism on) the regular $m$-differentials
of pairs.
\end{lemma}

\paragraph{Restrictions to lc centers under generic flop.}
$$
(g^*\omega)\rest{Y_t}=g\rest{Y_t}^*(\omega\rest{Y_s}).  (4)
$$

\begin{prop}\label{flops_in_fibers}
Let $(X/T,B)$ be a projective \tdlt\ family of $0$-pairs
with the generic connected klt fiber, horizontal $B$ and irreducible $T$, and
$g\in\Bir(X\to T/k,B)$ be a generic flop.
Then for any $t\in T$ there exists $s\in T$ such that,
for any minimal lc centers $(Y_t,B_{Y_t}),(Y_s,B_{Y_s})$ of
$(X_t,B_{Y_t}),(X_s,B_{Y_s})$ (even on blowups)
respectively there exists a log flop
$g_{Y_t}\colon (Y_t,B_{Y_t})\dashrightarrow (Y_s,B_{Y_s})$,
which satisfies
$$
\rest{Y_t}g^*=(g_{Y_t})^*\rest{Y_s}.
$$

If $(X_t,B_{X_t}),(X_s,B_{X_s})$ belong to an mp-trivial subfamily,
then, for any even natural number $m$,
$$
\rest{Y_t} g^*=(g_{Y_t})^*c^*\rest{Y_t}\colon
H^0(X,\omega_{X/T}^m[mB])\to H^0(Y_t,\omega_{Y_t}^m[mB_{Y_t}]).
$$
where $c^*:H^0(Y_t,\omega_{Y_t}^m[mB_{Y_t}])\to
H^0(Y_s,\omega_{Y_s}^m[mB_{Y_s}])$ is the canonical identification.
\end{prop}

\begin{proof}
We use induction on $\dim T$.
If $\dim T=0$, then by our assumptions $s=t, X=X_s=X_t=Y_s=Y_t,
\rest{Y_t}=\id_{X},g_{Y_t}=g,c^*=\id_{H^0(X,\omega_{X/T}^m[mB])}$
and
$$
\rest{Y_t} g^*=(g_{Y_t})^*\rest{Y_s}=
(g_{Y_t})^*c^*\rest{Y_t}=g^*\colon
H^0(X,\omega_{X/T}^m[mB])\to H^0(Y_t,\omega_{Y_t}^m[mB_{Y_t}]).
$$

Now suppose that $\dim T\ge 1$.

Construction of $g_{Y_t}$.
Take a rather general curve $C\subseteq T$ through $t$.
Such a curve means a birational on image morphism
$h\colon C\ni o\to T$ with $h(o)=s$.
Birationally, $C=h(C),g(C)=g_Th(C)$.
This gives a curve $g(C)$ with the morphism $gh\colon C\ni o\to T$
and $s=hg(o)$.
Moreover, the birational map $g_T\rest{C}\colon C\dashrightarrow g(C)$
induces (restriction) a flop of two \tdlt\ families of $0$-pairs over $C$:
$$
g\rest{X_C}\colon (X_C/C\ni o,B_{X_C})\dashrightarrow (X'/C\ni o,B_{X'}).
$$
The first family is pulling back for $h$, the second one for $gh$.
Actually, both families are normal \tdlt\ over $C$.
By construction $g\rest{X_C}$ is birational and
by Lemma~\ref{flopping_lc_center} is a flop.

Suppose first that $\dim T=1$ and $T=C,t=o,X_C=X,B_{X_C}=B$.
Add vertical boundaries $X_o,X_o'$.
Take any minimal lc centers
$Y_t\subset (X_o,B_{X_o})=(X_t,B_{X_t}),
Y_s\subset (X_o',B_{X_o'})=(X_s,B_{X_s})$.
Even we can suppose that they are centers on a \tdlt\ blowup
of fibers over $o$.
So, we replace both families by such blowups.
We can suppose also that $g$ is defined in
a minimal lc center $Y_t'\subset (X_o,B_{X_o})=(X_t,B_{X_t})$.
Then by Lemma~\ref{flopping_lc_center}
$g(Y_t')=Y_s'\subset (X_o',B_{X_o'})=(X_s,B_{X_s})$ is also
a minimal lc center and $g$ gives the flop
$$
g\rest{Y_t'}\colon (Y_t',B_{Y_t'})\dashrightarrow (Y_s',B_{Y_s'}).
$$
Let
$$
c_t\colon (Y_t',B_{Y_t'})\dashrightarrow (Y_t,B_{Y_t}),
c_s\colon (Y_s',B_{Y_s'})\dashrightarrow (Y_s,B_{Y_s})
$$
be canonical flops between minimal lc centers.
Then $g_{Y_t}=c_sg\rest{Y_t'}c_t\1$.

The real induction is for $\dim T\ge 2$.
For any minimal lc centers
$Y_t\subset (X_o,B_{X_o})=(X_t,B_{X_t}),
Y_s\subset (X_o',B_{X_o'})=(X_s,B_{X_s})$, even on \tdlt\ resolutions,
we construct minimal lc centers $Y_t'\subset (X_o,B_{X_o})=(X_t,B_{X_t}),
g(Y_t')=Y_s'\subset (X_o',B_{X_o'})=(X_s,B_{X_s})$ with a flop
given $g\rest{Y_t'}$ by a restriction.
Indeed, after a log resolution of the base extending $\Delta_T$, we can
suppose that $C$ is nonsingular and intersects transversally
the log structure.
Then $t=o$ and we can canonically identify the fiber over $o$ with
the fiber over $s$ before the resolution.
Take a nonsingular divisor $D\subset X$ extending the log structure $\Delta_T$ on $T$
and passing through $C$.
The mapping $g_T$ in general is not a log flop with respect to the log structure.
Nonetheless, after adding to $D+\Delta_T$ the preimage $g_T\1 \Delta_T$ and
adding to $\Delta_T$ the image $g_T(D+\Delta_T)$, we can convert
the birational automorphism $g_T$ into a regular flop (usually, not an autoflop)
$g_T\colon (T,\Delta_T)\to (T',\Delta_{T'})$
using a log resolution of the log map (torification).
So, $D$ after those modifications is still
a nonsingular divisor of $\Delta_T$, $C\subset D$ and
$D'=g_TD$ is also a nonsingular divisor of $\Delta_{T'}$.
Then by Lemma~\ref{flopping_lc_center} we constructed
a log flop
$$
g\rest{X_D}\colon (X_D/D,B_{X_D})\dashrightarrow(X_{D'}/D',B_{X_{D'}'}).
$$
Moreover, $g\rest{D}\rest{C}$ is the above flop over $C\to C'=g_TC$.
By induction we constructed required centers $Y_t'$ and $Y_s'$ and
a flop $g_{Y_t}=c_sg\rest{Y_t'}c_t\1$ as above.

Restrictions to lc centers under generic flop, (4),
implies a similar relation for $g_{Y_t}$:
$$
\rest{Y_t} g^*=(c_t\1)^*\rest{Y_t'}g^*=
(c_t\1)^*(g\rest{Y_s'})^*\rest{Y_s'}=
(c_t\1)^*(g\rest{Y_s'})^*c_s^*\rest{Y_s}=(g_{Y_t})^*\rest{Y_s}.
$$

If $(X_t,B_{X_t}),(X_s,B_{X_s})$ belong to an mp-trivial subfamily,
then
$\rest{Y_s}=c^*\rest{Y_t}$ by Lemma~\ref{restr_in_mp_triv} and
the required relation holds:
$$
\rest{Y_t} g^*=(g_{Y_t})^*\rest{Y_s}=
(g_{Y_t})^*c^*\rest{Y_t}.
$$
\end{proof}

\section{Interlacing}

\begin{ex}[Rational lc conic bundle structure] \label{nonunique_cb}
Let $X=C\times \PP^1,B=D_1+D_2,D_i=C\times y_i, i=1,2$,
where $C$ is a nonsingular curve and
$y_1,y_2\in\PP^1$ are two distinct points.
Then the conic bundle $f\colon C\times\PP^1\to C, (x,y)\mapsto x$,
has two horizontal sections $D_i$ in the reduced boundary $B$
and $K+B\equiv 0/C$.
A proper rational conic bundle structure $X\dashrightarrow C'$
with horizontal $B$ and $(K+B.F)=0$ for generic fiber $F$
the conic bundle is unique and coincide with the above one.
By a proper rational conic bundle on $X$  we mean
a rational conic bundle corresponding to
an imbedded family of rational curves [pencil] without
fixed points.
In the surface case such a conic bundle is [always] a regular pencil.
If the genus of $C$ is $\ge 1$ then the uniqueness
follows from rationality of $F$.
Otherwise by adjunction $0=(K+B.F)=(f^*K_C.F)=
(K_C.f(F))$ implies that
$F$ is again vertical.

However for $C=\PP^1$ there are infinitely many rational
(nonproper) conic bundles
on $X$ such that $B$ is horizontal on their regular
model and consists of two sections.
For instance this holds for general pencil
$P\subset \linsys{x\times \PP^1+\PP^1\times y}$
of conics through two generic points of $X$.
(This pencil is proper after a blowup of two fixed points.)
\end{ex}

\begin{prop}\label{unique_cb}
Let $(X,B+D)$ be a
projective plt wlc pair and
$D$ be the reduced part of $B+D$ with $2$
components.
Then there exists birationally at most one
proper rational conic bundle structure on $X$
such that $D$ is the double section and
$B$ does not intersect the generic fiber $F$ of
conic bundle, that is, $(K+B+D.F)=(K+D.F)=0$.
More precisely,
the conic bundle structure is birationally independent of
a plt wlc model of $(X,B+D)$.
\end{prop}

\begin{proof}
A rational conic bundle of $X$ is a rational contraction
$X\dashrightarrow T$ such that its generic fiber $F$
is a rational curve.
The conic bundle is proper if
it is regular near $F$, that is,
the generic fiber is a free curve.
Suppose that $(K+B+D.F)=(K+D.F)=0$.
Equivalently, $(K.F)=-2,(D.F)=2, (B.F)=0$.
The last condition means that $\Supp B\cap F=\emptyset$.
We prove that such a conic bundle is birationally
unique, that is, the generic fiber $F$ is unique.
Note also that if such a conic bundle structure is
proper on some plt wlc model of $(X,B+D)$, then
it will be proper on $X$ after a crepant blowup
(flop).
Thus it is sufficient to establish the uniqueness
on one fixed model $(X,B+D)$.

Step 1. Reduction to the case of a $0$-pair $(X,B+D)$.
Let $(X,B+D)\to X\lcm$ be an Iitaka contraction.
Then the generic fiber $F$ is contractible to a point on $X\lcm$.
Thus the uniqueness of conic bundle is sufficient to
establish on the generic log fiber $(X_\eta,B_{X_\eta}+D_\eta)$
where $\eta\in X\lcm$ is the generic point.
By construction $(X_\eta,B_{X_\eta}+D_\eta)$ is a $0$-pair.

Step 2. Reduction to the case when $B=0$.
Use the LMMP for $(X,(1+\varepsilon)B+D),0<\varepsilon \ll 1$.
Since the Kodaira dimension of $(X,(1+\varepsilon)B+D)$ is
$\ge 0$, the LMMP terminates with a wlc model.
Note also that the LMMP requires an appropriate initial model
with $\R$-Cartier $B$.
Such a model can be constructed as a small modification of $(X,B+D)$,
a $\Q$-factorialization of $B$.
Any small modification of $X$ is a flop of $(X,B+D)$ and
does not touch any generic fiber $F$.
For sufficiently small $\varepsilon$,
$(X,(1+\varepsilon)B+D)$ is also plt.
The divisorial contractions of the LMMP does not touch
$F$ because they are negative with respect to $B$ and
their exceptional locus lies in $\Supp B$.
For wlc $(X,(1+\varepsilon)B+D)$, $B$ is semiample.
Let $(X,(1+\varepsilon)B+D)\to X\lcm$ be the corresponding
Iitaka contraction.
Then as in Step 1 $F$ is contractible to a point by the contraction and
it is sufficient to verify the uniqueness for
the generic log fiber $(X_\eta,B_{X_\eta}+D_\eta)=(X_\eta,D_\eta)$.
By construction $B_{X_\eta}=0$.

Step 3. Reduction to the case when $\Diff_D 0=0$.
Use the canonical covering (fliz) of $(X,D)$.
Indeed, the canonical covering makes
$(X,D)$ a log Gorenstein $0$-pair, that is, $K+D\sim 0$.
Thus $(D,\Diff_D0)$ is also a log Gorenstein $0$-pair.
The plt property of $(X,D)$ gives the klt property
of $(D,\Diff_D0)$, and the canonical one in the Gorensten case.
Hence $\Diff_D 0=0$.

Note also that every proper rational conic bundle
gives a similar bundle on the covering.
Indeed, every proper rational fibration induces
the proper rational fibration on the covering.
The latter fibration is the rational contraction
for a Stein decomposition of
composition of the covering with former
contraction.
The generic fiber $F$ on $X$ is $\PP^1$
with a transversal double section $D$.
The divisorial ramification of the canonical covering
is only in $D$.
Thus the conic bundle fibration goes into
a conic bundle with the double section induced
by $D$.
Each section $D_i$ goes into a section of
the fibration after covering.

The mapping of fibrations for coverings is monomorphic.

Step 4. Final. Since $\Diff_D 0=0$,
then $D$ is rationally disconnected (separably
in the positive characteristic), that is,
two generic points of $D$ are not connected by
a rational curve on $D$.
The same holds for each component $D_i, i=1,2$, of $D$.
So, the base $T$ of any
rational contraction $X\dashrightarrow T$
with rational sections $D_i$ is also
rationally disconnected.
The base $T$ is birationally isomorphic to each of $D_i$.
Thus a rational proper conic bundle on $X$ is
a rational contraction given by
the rational connectedness.
\end{proof}

The rational disconnectedness of $D$ was important
in the last step of proof.

\begin{ex}\label{nonunique_cb2}
Let $X=\PP^1\times \PP^1$ and
$D\in \linsys{-K}$ be a smooth anticanonical curve.
Then $X$ has two conic bundle fibrations with the double section $D$.
Actually, for any double covering $D\to \PP^1$,
there are a wlc model of $(X,D)$ and
a conic bundle inducing this double covering.
\end{ex}

\begin{thm}\label{can_0-perest}
Let $(X,B+D)$ be a plt pair
with a plt wlc model $(Y,B_Y+D_Y)/X\lcm$ and
$D$ be the reduced part of $B+D$.

If $D$ is not vertical over $X\lcm$ then
there exists such a (projective plt) wlc model $(Y,B_Y+D_Y)/X\lcm$
that $X\dashrightarrow Y$ is a $1$-modification and
the model has a Mori log contraction $(Y,B_Y/T/X\lcm)$ with
$K_Y+B_Y+D_Y\equiv 0/T/X\lcm$.
In this case $D_Y$ has at most
$2$ irreducible components and each component of $D_Y$
is horizontal with respect to $Y\to X\lcm$.

In the case with $2$ horizontal components
$D_1=D_{1,Y},D_2=D_{2,Y}$ of $D_Y$ over $X\lcm$,
the Mori log contraction $Y\to T$ is a conic bundle.
The divisors $D_1,D_2$ are
rational sections of the conic bundle.
Such a conic bundle structure $Y\to T$ is
birationally unique for $(X,B+D)$.
More precisely, the conic bundle structure is
independent on a plt wlc model $(Y,B_Y+D_Y)$.

In the case with $2$ horizontal components,
if $(X,B+D)$ is wlc itself, then $D_Y$ is
the birational transformation of $D$
the conic bundle gives a generalized canonical log flop
$c\colon (D_1,B_{D_1})\dashrightarrow (D_2,B_{D_2})$
as the composition
\begin{align*}
(D_1,B_{D_1})\hookrightarrow (X,B+D)&\dashrightarrow
(Y,B_Y+D_Y)\twoheadrightarrow (T,(B+D)_T)
\twoheadleftarrow (Y,B_Y+D_Y)\\
&\dashleftarrow (X,B+D)
\hookleftarrow (D_2,B_{D_2}),
\end{align*}
where $(B+D)_T=B\sdiv$ is the divisorial part of adjunction
with respect to the conic bundle.

The canonical property in addition to uniqueness means
that the restrictions of differentials (Poincare residues)
(super)commutes with the flop:
$$
c^*\rest{D_2}=(-1)^m\rest{D_1}\colon
H^0(X,\omega^m[mB])=H^0(Y,\omega_{Y}^m[mB_Y])\to
H^0(D_1,\omega_{D_1}^m[mB_{D_1}]).
$$

The induced standard structure
is preserved under the flop.

If $B+D=B\st+B\Cr+D$ is the standard structure with
$\Q$-mobile $B\Cr$ the flop preserves the induced
standard structure.
\end{thm}

\begin{proof}
By the LMMP a wlc model $(Y,B_Y+D_Y)/X\lcm$ exists
exactly when the Kodaira dimension of $(X,B+D)$
is $\ge 0$.
In particular, this will be a generalized flop if $(X,B+D)$ is wlc.
One can suppose that $X\dashrightarrow Y$ is
$1$-modification.
Indeed, to apply the LMMP we need a projective lc model,
e.g., a log resolution with boundary multiplicities $1$ in
exceptional divisors.
By the plt property all such divisors will be contracted
and $(Y,B_Y+D_Y)$ will be plt.

If $D$ is not vertical over $X\lcm$ then
one can apply the LMMP to $(Y/X\lcm,B_Y)$ assuming
that $D_Y$ is $\Q$-factorial.
(The latter needs a small $\Q$-factorialization.)
This gives a required
Mori log contraction $(Y,B_Y/T/X\lcm)$ such that
$K_Y+B_Y+D_Y\equiv 0/T/X\lcm$.
The generic log fiber $(Y_\eta,B_{X_\eta}+D_\eta)$ of $Y/T$ has
dimension $\ge 1$ and is a plt $0$-pair.
Thus $D_Y$ has at most
$2$ irreducible components and each component of
$D_Y$ is horizontal with respect to $Y\to X\lcm$.

By the plt property of $(Y,B_Y+D_Y)$,
$D_Y$ is
the birational transformation of $D$ if $(X,B+D)$ is wlc.

In the case with $2$ horizontal components $D_1=D_{1,Y},D_2=D_{2,Y}$ of $D_Y$
over $X\lcm$,
the Mori log contraction $Y\to T$ is a conic bundle.
The divisors $D_1,D_2$ are
rational sections of the conic bundle.
Such a conic bundle structure $Y\to T$ is
birationally unique for $(X,B+D)$.

If $(X,B+D)$ is wlc then
the conic bundle gives a generalized log flop
$(D_1,B_{D_1})\dashrightarrow (D_2,B_{D_2})$
as the composition
\begin{align*}
(D_1,B_{D_1})\hookrightarrow (X,B+D)&\dashrightarrow
(Y,B_Y+D_Y)\twoheadrightarrow (T,(B+D)_T)
\twoheadleftarrow (Y,B_Y+D_Y)\\
&\dashleftarrow (X,B+D)
\hookleftarrow (D_2,B_{D_2}),
\end{align*}
where $(B+D)_T=B\sdiv$ is the divisorial part of adjunction
for the conic bundle.
The moduli part of adjunction is trivial.

The induced standard structure
is preserved under the flop.

If $B+D=B\st+B\Cr+D$ is the standard structure with
$\Q$-mobile $B\Cr$ the flop preserves the induced
standard structure as for b-divisors.
If one would like to have the last property for divisors,
then one needs to assume that $B\Cr\equiv 0/T$ and is
a divisor.
This holds after a (possibly not small)
modification of the conic bundle over $T$.

The statement and formula, which relates restrictions with
the canonical flop follows from Lemma~\ref{sum_resid},
because restrictions preserve log regular $m$-differentials.

Finally, the uniqueness of flop follows from Proposition~\ref{unique_cb}.
Indeed, after a crepant blowup of divisors on $(Y,B_Y+D_Y)$ with
log discrepancies $\le 1$ one can assume that
a fixed rational conic bundle structure on another wlc model
of $(X,B+D)$ with horizontal $D$ is proper on $Y$.
\end{proof}

A typical example of an interlaced triple
comes from a \tdlt\ triple.

\begin{ex}[Triple of minimal lc centers]\label{tripl_min_c}
Let $(X,B)$ be a \tdlt\ pair and
$(Y,B_Y)=(X,B)\mc$ be its pair of minimal lc centers.
Then $(Y,B)$ is also \tdlt, respectively, wlc, standard etc,
if so does $(X,B_Y)$.

For (projective) wlc $(X,B)$,
there is an interlacing on $(Y,B_Y)$.
The vertexes of $\Gamma$ are irreducible components $Y_i$
of $Y$.
An edge between $Y_i,Y_j$ is
an invariant (with respect to the log structure)
closed irreducible subvariety $Z\subseteq X$,
a flopping center, such that
\begin{description}

\item{}
$Y_i,Y_j\subset Z$ are invariant divisors and

\item{}
there exists a rational $0$-contraction of $(Z,B_Z)$
with horizontal divisors $Y_i,Y_j$.
Equivalently, there is a free curve $C\subseteq Z$ (the generic fiber
of $0$-contraction) such that $(K_Z+B_Z.C)=0$ and
$(Y_i,C),(Y_j.C)>0$.
\end{description}
Indeed, in this situation there exists a
generalized log flop $(Y_i,B_{i,Y})\dashrightarrow (Y_j,B_{j,Y})$
by Theorem~\ref{can_0-perest}.

Note that, for $i=j$, one can get sometimes an autoflop
$(Y_i,B_{i,Y})\dashrightarrow (Y_i,B_{i,Y})$, an involution
in $\Bir(Y_i,B_{i,Y})$.
Unfortunately, it is not unique by Example~\ref{nonunique_cb2} and
so is not very useful in general.
So, we agree that, for $i=j$,
a flopping center has a single invariant divisor
$Y_i=Y_j$ and the flop is identical.
\end{ex}

\begin{lemma}\label{chain_of_flop}
Let $(X/T,B)$ be a \tdlt\ $0$-contraction
with a boundary $B$.
Then any two minimal lc center over $T$
can be connected by flopping centers over $T$.
\end{lemma}

\begin{proof}
The proof is similar to the case over a point: $T=\pt$.
Let $Y,Y'$ be two minimal lc center over $T$.
We will find a chain of flopping centers
$C_1,\dots,C_n$ on $X$ such that
$Y\subset C_1,\dots, Y'\subset C_n$.
As usually, a chain means that $C_i$ intersects $C_{i+1}$
for $1\le i\le n-1$.

Step 1. We can suppose that $X$ is irreducible.
Otherwise, take the normalization
$(X\nor,B\nor)=\coprod (X_i,B_i)$.
Note that $X_i$ is possible not geometrically irreducible and
not connected (fiberwise) over $T$.
Nonetheless, since $X/T$ is contraction,
the fibers are connected.
Thus there is chain of components $X_i$:
$$
X_1\supset Y_1\subset X_2\supset
\dots \subset X_{n-1}\supset Y_{n-1}\subset X_n
$$
with common minimal lc centers $Y_i$
for  each pair $(X_i,B_i),(X_{i+1},B_{i+1}), 1\le i\le n-1$.
Hence it is sufficient to find a chain of
flopping centers for each pair of minimal
centers $Y_{i-1},Y_{i}\subset X_i,1\le i\le n$,
where $Y_0=Y$ and $Y_n=Y'$.
So, we replace $(X/T,B)$ by $(X_i/T_i,B_i)$,
where $X_i\to T_i$ a $0$-contraction given
by a Stein decomposition.
It is \tdlt.

Step 2.
Dimensional induction.
The case $\dim X/T=0$ is empty.
The case $\dim X/T=1$ is flopping.
Indeed, the generic fiber is
a (geometrically) irreducible curve
with at most two minimal centers.
If there are no centers, then the statement is
empty.
If there are centers, then $X$ itself is flopping.
A chain is trivial: $C_1=X$.

If $\dim X/T\ge 2$ and
there are minimal lc centers $Y,Y'$,
then two situations are possible:
\begin{description}

\item{\rm (1)}
the generic fiber $(X_\eta,B_{X_\eta})$ is nonklt, but plt, or

\item{\rm (2)}
all proper lc centers over $T$ are lc
connected, that is, their union is connected over $T$.
\end{description}
In (1), the chain is trivial as above: $C_1=X$.
In (2), we apply induction to the \tdlt\ pair
$(Y/T,B_Y)$, where $Y$ is the union of all
invariant divisors over $T$.
Note that the lc centers for a \tdlt\ pair
are invariant subvarieties.
\end{proof}

\begin{cor}
Let $(X,B)$ be a connected projective \tdlt\ $0$-pair.
Then $(X,B)\mc$ is a connected interlaced $0$-pair.
\end{cor}

\begin{proof}
Immediate by Theorem~\ref{can_0-perest},
Example~\ref{tripl_min_c} and
Lemma~\ref{chain_of_flop} over a point: $T=\pt$.
\end{proof}

By the uniqueness or a canonical construction
of flops in Theorem~\ref{can_0-perest},
we can make interlacing for families.

\begin{cor}\label{interl_of_famil}
Let $(X/T,B)$ be a family
connected projective \tdlt\ $0$-pairs.
Then, for any generic point $\eta$ of $T$,
$(X/T,B)\mc$ is
a connected interlaced family of \tdlt\ $0$-pairs.
\end{cor}

\begin{proof}
Immediate by Theorem~\ref{can_0-perest},
Example~\ref{tripl_min_c} and
Lemma~\ref{chain_of_flop}.

By general properties of \tdlt\ families,
$(X/T,B)\mc=(X\mc/T,B_{X\mc})$ is
a \tdlt\ family of $0$-pairs.
Irreducible components $Y\subseteq X\mc$
are not necessarily geometrically irreducible or/and
connected (fiberwise) over $T$.
However, they are connected by flopping centers
according to Lemma~\ref{chain_of_flop}.
On the other hand, each flopping center,
with a geometrically reducible lc center $Y$ over $T$,
determines a rational conic bundle and flop of
centers by Theorem~\ref{can_0-perest}.
(Actually, $Y$ can be irreducible, but with
two components in generic geometric fibers.)
If $Y$ is geometrically irreducible, then
the flop is identical on $Y$.
For this constructions, it is sufficient to
consider the generic fiber $(X_\eta,B_{X_\eta})$.
To apply the theorem, take an algebraic closure
of $\eta$.
By the uniqueness required flops are defined over $\eta$
and by definition over $T$.
\end{proof}

\begin{ex}[Blowup of an lc center]\label{blowup_lc_cent}
Let $f\colon (\tilde{X},B_{\tilde{X}})\to (X,B)$ be
a crepant extraction (flop) of an lc center $f(D)$
with a prime divisor $D\subset X$.
We suppose that $f(D)$ is a real lc center,
that is $B\ge0$ and $(X,B)$ is lc
near the generic point of $f(D)$.
So, $B$ is a boundary near the center.
By definition $D$ is a reduced divisor in $B_{\tilde{X}}$ and
$(D,B_D)$ is lc with a boundary $B_D$ (generically) over the center.
The mapping $f\rest{D}\colon (D,B_D)\to f(D)$ is
a $0$-contraction.
So, any two minimal lc center of $(D,B_D)$ are
related by a flop according to Corollary~\ref{interl_of_famil}
and to the connectedness of lc locus.
Such a minimal center always exists, possibly,
$D$ itself.
If $(X,B)$ is \tdlt\ along $f(D)$, then, for
any minimal lc center $Y\subseteq (D,B_D)$ over $f(D)$,
the restriction
$$
f\rest{Y}\colon (Y,B_Y)\to (f(Y),B_{f(Y)})
$$
is birational and a flop.

However, if $(X,B)$ is slc along $f(D)$ and
$f$ is a normalization with a blowup, then
$f\rest{Y}$ can be a fliz, if it is generically finite, e.g.,
$2$-to-$1$ for osculation in divisors of slc $(X,B)$.
\end{ex}

\section{Relative b-representation}

This section gives results which are
relative analogues and generalizations of a well-known finiteness
of representation of flops on differentials
(Nakamura, Ueno, Sakai, Fujino, etc)
\cite{NU} \cite{S} \cite{U} \cite{FC} (the last preprint has
historic remarks on the question).

\begin{df}
A linear representation $G\to\Aut V$ is {\em finite\/},
if so does its image.
An {\em order} of the representation is
the order of its image.
The same works for projective representations $G\to \PP(V)$.
\end{df}

Even for sheaves $\omega_{X/T}^m[mB]$ and $\sO_X(mM)$, which are isomorphic,
the representation of flops on $H^0(X,\omega_{X/T}^m[mB])$ is typically
different from that of on $H^0(X,m\sM)$.
For example, if $X$ is a K3 surface then the representation of
the automorphisms of $X$ on $H^0(X,\sO_X)$ is trivial, and $\omega_X\cong \sO_X$, but
the representation on $H^0(X,\omega_X)$ can be nontrivial.
In this situation,
the automorphisms with trivial action on $H^0(X,\omega_X)$ are
known as symplectic.
In general, the difference between representations on
global sections for isomorphic invariant invertible sheaves is in scaler matrices.
So, the projective representations of isomorphic invariant invertible sheaves or
of invariant up to linear equivalence divisors coincide.
Thus in the proof of Corollary~\ref{generic_Kawamata} it does not matter
a choice of a moduli part of adjunction $\sM$ as a sheaf or as a divisor.
It is important only the flop invariance of $\sM$.
But in Theorem~\ref{flop_repres} the canonical choice is an important assumption.

\begin{ex}[Toric representation]\label{toric_repr}
Take a log pair $(\PP^1,0+\infty)$.
This is a toric variety with an action
$tx,t\in T^1=k^*$, where $x$ is a nonhomogenous coordinate.
For the sheaf $\sO_{\PP^1}(n(\infty-0))$,
the representation of $T$ is $t^nx^n, H^0(\PP^1,n(\infty-0))=kx^n$
has weight $n$.
The sheaf $\sO_X$ with $n=0$ has trivial representation and is
isomorphic canonically to $\omega_{\PP^1}(0+\infty), 1\mapsto dx/x$.

Similarly, it is easy to construct for any rank $1$ invariant sheaf an isomorphic
invariant sheaf with
infinite representation on its global section if the group of its
isomorphisms is infinite and if there exists a nonconstant rational
function with the invariant divisor.
However those log canonical divisors and functions exist only on nonklt
pairs.
So, for the klt pairs, any scalar representation
is finite, and
the finiteness of a linear representation of
an invertible sheaf is the property of all class of isomorphic sheaves.
\end{ex}

\begin{lemma}\label{prod_of_repres}
Let $\sD_1,\sD_2$ be two b-divisors, which are
effective up to linear equivalence and
invariant up to linear equivalence with
respect to action of a group
$G\subseteq\Bir(X)$ of birational automorphisms.

(1) Then $\sD_1+\sD_2$ is also invariant up to linear
equivalence and
the finiteness of projective representation of $G$ on
$\PP(H^0(X,\sD_1+\sD_2))$ implies the same for
representations of $G$ on $\PP(H^0(X,\sD_1)),\PP(H^0(X,\sD_2))$.
Moreover, the orders of both representations
are bounded by and divide the order of
representation on $\PP(H^0(X,\sD_1+\sD_2))$.

(2) The converse holds if sections for $\sD_1$ and $\sD_2$
generate the sections of $\sD_1+\sD_2$, that is,
the surjectivity
$$
H^0(X,\sD_1)\otimes H^0(X,\sD_2)\twoheadrightarrow H^0(X,\sD_1+\sD_2),
s_1\otimes s_2\mapsto s_1s_2,
$$
holds.
The order of representation on $\PP(H^0(X,\sD_1+\sD_2))$
is bounded by and divide
the product of orders of representations on
$\PP(H^0(X,\sD_1)),\PP(H^0(X,\sD_2))$.

(3) For a natural number $m>0$, if sections for $\sD_1$ generate
the sections for $mD_1$, then
the representations on $\PP(H^0(X,\sD_1)),\PP(H^0(X,m\sD_1))$ have
isomorphic images and the same orders.

(4)
If the divisors $\sD_1,\sD_2$ are invariant
with respect to $G$,
then (2) holds for linear representations.
If $\sD_1$ is invariant, then in (3)
the image of representation on $H^0(X,m\sD_1)$ is
a quotient  of the image for $H^0(X,\sD_1)$ and
the order of the former representation divides
the order of the letter one.

(5) The statements (1-4) also holds for $G$-invariant
b-sheaves instead of b-divisors, even $G$-invariant
up to isomorphism for the projective representations.
\end{lemma}

In general (1) does not hold for linear representations.

\begin{ex}
For any natural number $n>0$,
the representations of $k^*$ on $H^0(\PP^1,n(\infty-0))$ and
on $H^0(\PP^1,-n(0-\infty))$ (see Example~\ref{toric_repr})
are infinite, but the representation on
their product $H^0(X,0(\infty-0))$ is trivial.
\end{ex}

\begin{lemma}\label{repres_comparable}
Let $G\subseteq\Bir(X)$ be a group of birational automorphisms, and
$\sM,\sD$ be two b-divisor on $X$ such that
\begin{description}

\item{\rm (1)}
$\sM,\sD$ are invariant up to linear equivalence
with respect to $G$,

\item{\rm (2)}
$\sM$ is semi-ample, and

\item{\rm (3)}
$\sD\equiv r\sM$ for some real number $r\ge 0$.

\end{description}
Then the finiteness of representation of $G$ on
$\PP(H^0(X,m\sM))$ for sufficiently large natural numbers $m$
implies the finiteness of representation on
$\PP(H^0(X,\sD))$.
A bound for the last representation is the same as
for $\PP(H^0(X,m\sM))$.
\end{lemma}

\begin{lemma}\label{resid_injection}
Let $(X/T,B)$ be a \tdlt\ family of connected $0$-pairs, and
$Y\subseteq X\mc$ be a component of its minimal lc center over $T$.
Then, for any natural number $m$, the Poincare residue gives
a canonical inclusion
$$
H^0(X,\omega_{X/T}^m[mB])\subseteq H^0(Y,\omega_{Y/T}^m[mB_Y]),
\omega\mapsto \omega\rest{Y}=\res_Y\omega.
$$
\end{lemma}

\begin{lemma}\label{restr_in_mp_triv}
Let  $(X/T,B)$ be a connected [\tdlt] mp-trivial reduced family
of connected \tdlt\ $0$-pairs with a horizontal boundary $B$ and
$m$ is a natural number
such that $\sM\sim_m0$
and $m$ is even.
Then for any $t,s\in T$ there exist canonical
identifications given restrictions and residues:
$$
H^0(X,\omega_{X/T}^m[mB])=H^0(X_t,\omega_{X_t}^m[mB_{X_t}])=
H^0(X_s,\omega_{X_t}^m[mB_{X_s}])=
$$
$$
H^0(Y_t,\omega_{Y_t}^m[mB_{Y_t}])=H^0(Y_s,\omega_{Y_s}^m[mB_{Y_s}]),
$$
where $Y_t,Y_s$ are the minimal lc centers or even
the minimal lc centers for \tdlt\ blowups.
If $c^*:H^0(Y_t,\omega_{Y_t}^m[mB_{Y_t}])\to
H^0(Y_s,\omega_{Y_s}^m[mB_{Y_s}])$ denotes
the last canonical identification, then
$c^*\rest{Y_t}=\rest{Y_s}$.
\end{lemma}

\begin{lemma}\label{flopping_lc_center}
Let $g\colon (X,B)\dashrightarrow (Y,B_Y)$
be a flop of \tdlt\ pairs, and
$Z\subseteq X$
be a lc center of $(X,B)$ such that $g$ is defined in $Z$.
Then $g(Z)\subseteq Y$ is also a lc center of $(Y,B_Y)$ and
$g\rest{Z}\colon Z\dashrightarrow g(Z)$ is a rational contraction.
Moreover, if $g\rest{Z}$ is birational, then it is
a log flop
$$
g\rest{Z}\colon (Z,B_Z)\dashrightarrow (g(Z),B_{g(Z)}).
$$
In particular, if $Z$ is a minimal lc center, then $g(Z)$ is also minimal, and
$g\rest{Z}$ is a log flop.
\end{lemma}

\begin{proof}
By definition, we can suppose that both varieties $X,Y$ are irreducible.
Otherwise, we take irreducible component of $X$ containing $Z$ and
its image in $Y$.
Here we use the \tdlt\ property (no osculation).

Also by definition any lc center is an image of a b-divisor $D$
with the boundary multiplicity $1$ with respect to $(X,B)$.
Equivalently, there exists an extraction $\tilde{X}\to X$ of
$D$.
Since the singularities are \tdlt, we can make a flop,
a crepant \tdlt\ resolution
(even very economical with one exceptional divisor $D$).
By construction $D\subset \tilde{X}$ is a divisor with
a contraction $D\to Z$.

To verify that $g\rest{Z}\colon Z\dashrightarrow g(Z)$
is a rational contraction,
it is sufficient to verify that the composition $D\to Z\dashrightarrow g(Z)$
is a rational contraction.
Note for this that $D\dashrightarrow g(Z)$ is lc center for
$(Y,B_Y)$, because the composition
$(\tilde{X},B_{\tilde{X}})\to (X,B)\dashrightarrow (Y,B_Y)$
is also a flop of \tdlt\ pairs.
By the crepant property of flops,
the composition maps $D$ onto the lc center $g(Z)$.
This is a rational contraction by the \tdlt\ property of $(Y,B_Y)$,
there exists an extraction of $D$ in $Y$
with contraction onto $g(Z)$ as above.

The flopping property of birational $g\rest{Z}$
follows from the divisorial adjunction.
For this we use a dimensional induction.
It is the divisorial adjunction if $Z$ is a divisor.
If $Z$ is not a divisor, then by the \tdlt\ property
there exists an invariant divisor $W$ containing $Z$.
If $W\dashrightarrow g(W)$ is birational, then
we can use induction.
Otherwise we extract an invariant b-divisor $\tilde{W}\subset \tilde{Y}\to Y$ such that
$W$ maps to $\tilde{W}$.
Now the mapping of $Z$ to $\tilde{W}$ is not necessarily defined.
If so, then we blow up $Z$ in $W$ and
by Example~\ref{blowup_lc_cent} we can find
flop of a lc center $(\tilde{Z},B_{\tilde{Z}})\dashrightarrow (Z,B_Z)$
such that the mapping of $\tilde{Z}$ to $\tilde{W}$ is defined.
The composition $\tilde{Z}\dashrightarrow \tilde{W}\to Y$
maps $\tilde{Z}\dashrightarrow g(Z)=\tilde{Z}\to Z\dashrightarrow g(Z)$ and
gives a required flop by induction.

Finally, suppose that $Z$ is a minimal lc center.
Then the lc centers of contraction $(D,B_D)\to Z$ are
only horizontal (the connectedness of lc centers).
And vice versa.
By the \tdlt\ property, any minimal lc center $\tilde{Z}$ of $(D,B_D)$
gives a (flop) birational mapping to $Z$.
Thus $g(Z)$ is minimal and $Z\dashrightarrow g(Z)$ is birational.
\end{proof}

\begin{prop}\label{repres_exten}
Let $(X/T,B)$ be a \tdlt\ family of connected $0$-pairs, and
$Y\subseteq X\mc$ be an irreducible component of
its minimal lc center over $T$.
Then, for any natural number $m$,
each representation linear transformation
$g^*$ on $H^0(X,\omega_{X/T}^m[mB])$ can
be extended to a representation linear transformation $g_Y^*$ on
$H^0(Y,\omega_{Y/T}^m[mB_Y])$.
That is, for any $g\in\Bir(X\to T/k,B)$,
there exists $g_Y\in \Bir(Y\to T/k,B_Y)$ such that
$$
g^*=(g_Y^*)\rest{H^0(X,\omega_{X/T}^m[mB])},
$$
where $g_Y^*$ is the representation of $g_Y$ on $H^0(Y,\omega_{Y/T}^m[mB_Y])$.
\end{prop}

\begin{proof}
An extension can be done under the canonical inclusion
$$
V=H^0(X,\omega_{X/T}^m[mB])\subseteq H^0(Y,\omega_{Y/T}^m[mB_Y]),
\omega\mapsto \omega\rest{Y}=\res_Y\omega
$$
of Lemma~\ref{resid_injection}.

Step 1.
If a flop $g$ is defined in $Y$, then $g(Y)\subseteq X\mc$ and
is also an irreducible component.
In the case $g(Y)=Y$, $g$ induces a generic flop $g_Y=g\rest{Y}$ on $(Y,B_Y)$
by Lemma~\ref{flopping_lc_center}.
In this case, $(g_Y^*)\rest{V}=\rest{Y}g^*$ is a general
invariance of the Poincare residue.

Step 2.
More generally,
if a flop $g$ is defined in $Y$, but possibly $g(Y)\not=Y$,
then $g$ induces a log flop $g\rest{Y}\colon (Y,B_Y)\to (g(Y),B_{g(Y)})$
again by Lemma~\ref{flopping_lc_center}.
By connectedness of fibers and Lemma~\ref{chain_of_flop},
Theorem~\ref{can_0-perest},
there exists a chain $C_1,\dots,C_n,n\ge 1$, of
flopping centers $C_i$ on $X$ such that $Y\subset C_1,\dots,g(Y)\subset C_n$ and
the chain of centers define a sequence of (canonical) flops $Y=Y_0\dashrightarrow Y_1
\dashrightarrow\dots\dashrightarrow Y_{n-1}\dashrightarrow g(Y)=Y_n$.
(According to our agreement,
$Y_0=Y_1$ and/or $Y_{n-1}=Y_n$, if
the flopping centers $C_1$ and/or $C_n$ have respectively
a single minimal lc center.)
They are flops with respect to adjoint boundaries $(Y_i,B_{Y_i})$.
Their composition gives a flop $c\colon (Y,B_Y)\dashrightarrow (g(Y),B_{g(Y)})$,
canonical with respect to differentials.
The canonicity means that all such flops are identical
on restricted sections for every {\em even\/} $m$.
The flops agrees with restrictions (Poincare residues):
for every $c_i\colon (Y_i,B_{Y_i})\dashrightarrow (Y_{i+1},B_{Y_{i+1}})$,
$$
c_i^*\rest{Y_{i+1}}=\rest{Y_i},
$$
and the inclusion is given by the Poincare residue,
the identifications $=$ by canonical flops:
$$
H^0(X,\omega_{X/T}^m[mB])\subseteq H^0(Y,\omega_{Y/T}^m[mB_Y])=
H^0(Y_i,\omega_{Y_i}^m[mB_{Y_i}])=H^0(g(Y),\omega_{g(Y)/T}^m[mB_{g(Y)}]).
$$
(Usually, $Y/T$ is not geometrically irreducible and
the inclusion (Poincare residue) is proper.)
Thus, for $g_Y=c\1(g\rest{Y})\colon Y\to Y$,
$g_Y^*$ extends the representation of $g$ from
$H^0(X,\omega_{X/T}^m[mB])$
to $H^0(Y,\omega_{Y/T}^m[mB_Y])$.
Indeed, for any $\omega\in V$,
$$
(g_Y^*)(\omega\rest{Y})=(g\rest{Y})^*(c\1)^*(\omega\rest{Y})=
(g\rest{Y})^*(\omega\rest{g(Y)})=(g^*\omega)\rest{Y}.
$$

Step 3.
If a flop $g$ is not defined in $Y$ ($Y$ is in indeterminacy locus),
then we make a blowup $(X',B_{X'})\to (X,B)$ in $Y$.
For \tdlt\ families such a blowup exists.
However, in our situation, the problem is birational and
it is sufficient to consider $(X_\eta,B_{X_\eta})$.
In the generic case, the only problem is nonirreducibility of $X_\eta$.
In this case, we replace $X_\eta$ by its normalization
with isomorphism of gluing divisors and identification of
differentials along them.
Then a blowup on any component should be done simultaneously
in identified centers on both gluing divisors.
We identify the blown up centers, in particular,
their minimal lc centers.
Take any minimal lc center $Y'$ over $Y$.
Then the blowup gives a canonical log flop
$c'\colon (Y',B_{Y'})\to (Y,B_Y)$.
The canonicity again means the same as above:
$$
H^0(X',\omega_{X'/T}^m[mB_{X'}])=
H^0(X,\omega_{X/T}^m[mB])\subseteq H^0(Y,\omega_{Y/T}^m[mB_Y])
=H^0(Y',\omega_{Y'/T}^m[mB_{Y'}]).
$$
If $g'$ is $g$ on $X'$, and is defined in $Y'$,
then put $g_Y=c'c\1(g'\rest{Y'}){c'}\1$, where $c\colon Y'\to g(Y')$ is
now constructed for $g',Y',(X',B_{X'})$ as
above for $g,Y,(X,B)$.
In this situation $g_Y^*$ also extends the representation of $g$
from $H^0(X,\omega_{X/T}^m[mB])=H^0(X',\omega_{X'/T}^m[mB_{X'}])$
to $H^0(Y,\omega_{Y/T}^m[mB_Y])=H^0(Y',\omega_{Y'/T}^m[mB_{Y'}])$.

If $g'$ is not defined in $Y'$ we make the next blowup etc.
Finally, we associate, to each flop $g\in\Bir(X\to T/k,B)$,
a flop $g_Y\in\Bir(Y\to T/k,B_Y)$ with the same (sub)representation:
$$
g^*=(g_Y^*)\rest{H^0(X,\omega_{X/T}^m[mB])}
\text{ on } H^0(X,\omega_{X/T}^m[mB])\subseteq
H^0(Y,\omega_{Y/T}^m[mB_Y]).
$$
\end{proof}

A version of the Burnside theorem.

\begin{thm} \label{Burnside}
Let $G\subseteq \Aut V$ be a group of linear transformation
of a finite dimensional linear space $V$ over a field $k$ such that
\begin{description}

\item{\rm (1)} $G$ is torsion, that is, for every $g\in G$, there exists a
positive integral number $m$ such that $g^m=1$, and

\item{\rm (2)}
finitely generated or

\item{\rm (3)}
every element $g\in G$ is defined over a field $l_g$
which has a uniformly bounded degree over
a field of pure transcendent extension over the prime subfield in $k$.

\end{description}
Then $G$ is finite.
\end{thm}

For example, (3) holds if $l_g$ has a uniformly bounded degree over
a field $l\subseteq k$ of finite type over the prime subfield in $k$
which independent of $g$.

The following result is a special case of
Corollary~\ref{repr_lc_can} below.
Technically, this is the most crucial step.

\begin{thm} \label{flop_repres}
Let $(X_\eta,B_{X_\eta})$ be a generic wlc $0$-pair,
where $X_\eta$ is geometrically irreducible.
Then, for any natural number $m$,
the canonical representation
of generic log flops on differentials
$$
\Bir(X_\eta\to\eta/k,B_{X_\eta})\to
\Aut H^0(X_\eta,\omega_{X_\eta}^m[mB_{X_\eta}]), g \mapsto g^*,
$$
is finite.
Moreover, the order of representation
has a uniform bound, independent of $m$.
\end{thm}

\begin{proof}
Since the representation is independent of
a wlc model of $(X_\eta,B_{X_\eta})$,
we can suppose that the model $(X_\eta,B_{X_\eta})$
is projective.
To construct such a model one can use the LMMP.
Below we use some other modifications of this model and
even a completion over $k$.

We can suppose that $B$ is $\Q$-divisor.
If the letter does not hold then $H^0(X_\eta,\omega_{X_\eta}^m[mB_{X_\eta}])=0$
for every $m\not=0$ because $(X_\eta,B_{X_\eta})$ is a $0$-pair.

Step 1.
By Lemmas~\ref{prod_of_repres} and~\ref{repres_comparable},
we can suppose that $m$ is sufficiently divisible, and
the finiteness needed only for {\em some\/} such $m$.
It is enough to suppose that $\omega_{X_\eta}^m[mB_{X_\eta}]$
is invertible, equivalently the divisor
$m\sM$ is Cartier, where $\sM$ is
a canonical upper moduli part of adjunction,
and that $H^0(X_\eta,\omega_{X_\eta}^m[mB_{X_\eta}])$
(this space of section of a b-sheaf is finite dimensional)
generate the relative log canonical ring
$\alg (\omega_{X_\eta}^m[mB_{X_\eta}])$.
Indeed, by Lemma~\ref{prod_of_repres}, (5),
the representation of $\Bir(X_\eta\to\eta/k,B_{X_\eta})$
on $\alg (\omega_{X_\eta}^m[mB_{X_\eta}])$ is finite.
Thus it is finite in each degree $l$, that is,
on each $H^0(X_\eta,\omega_{X_\eta}^{lm}[lmB_{X_\eta}])$.
On the other hand, for every natural number $n$,
$nm\sM/m=n\sM$.
Therefore, by Lemma~\ref{repres_comparable}, the projective representation on
$\PP(H^0(X_\eta,\omega_{X_\eta}^n[nB_{X_\eta}]))$ is finite
with the same bound as for the algebra.
A difference with the linear representation on
$H^0(X_\eta,\omega_{X_\eta}^n[nB_{X_\eta}])$ is
only in scalar matrices on rather general fibers $(X_t,B_{X_t})$ (cf. Step 5 below).
Thus it is the $0$-dimensional version of the theorem: $\eta=k$.
By Proposition~\ref{repres_exten}, this case can be reduced
to the same statement for a klt $0$-pair $(Y,B_Y)$.
Take a minimal lc center $Y\subseteq X_{t,}\mc$, assuming that
$X_t$ is \tdlt.
But the required result for the klt pairs is well-known (cf. Step 6 below).
Note that after the reduction we need to consider all
flops of $(Y,B_Y)$ and their representations but with
scalar restrictions on $H^0(X_t,\omega_{X_t}^n(nB_{X_t}))$.
The final uniform bound is the maximum for two algebras
$\alg (\omega_{X_\eta}^m[mB_{X_\eta}])$ and
$\bigoplus_{m\ge0}H^0(Y,\omega_Y^m(mB_Y))$.

We suppose also that $m$ is {\em even\/} (see Step 2).

Additionally, we assume, that there exists a nonzero section
$\omega_0\in H^0(X_\eta,\omega_{X_\eta}^n[nB_{X_\eta}])$
vanishing on the birational reduced b-divisor $\sD$ of $\eta$,
which contains all centers in $\overline{\eta}$ of degenerations of $X_\eta$.
More precisely, $\Supp\sD$ contains all special points $t\in\overline{\eta}$
such that $\dk(X_t,B_{X_t})<\dk(X_\eta,B_{X_\eta})$.
Actually, it is sufficient for a subdivisor of $\sD$ related
to $\Delta$ in Step 4 below.

Step 2.
We can suppose that $(X_\eta,B_{X_\eta})$ is klt.
Equivalently, $\dk(X_\eta,B_{X_\eta})=\dim X_\eta$.
By the LMMP we can suppose that $(X_\eta,B_{X_\eta})$ is dlt.
If $(X_\eta,B_{X_\eta})$ is not klt then we consider
the minimal lc center $(X_{\eta,}\mc,B_{\eta,}\mc)$
(a  generic family of interlaced pairs).
It can have disconnected fibers (geometrically not irreducible).

Fix an irreducible component $Y\subseteq X_{\eta,}\mc$.
Then Lemma~\ref{resid_injection} gives a canonical inclusion
$$
H^0(X_\eta,\omega_{X_\eta}^m[mB_{X_\eta}])\subseteq H^0(Y,\omega_{Y/T}^m[mB_Y]),
\omega\mapsto \omega\rest{Y}=\res_Y\omega.
$$
On the other hand, by Proposition~\ref{repres_exten},
each representation linear transformation
$g^*$ of $H^0(X_\eta,\omega_{X_\eta}^m[mB_{X_\eta}])$ can
be extended to a representation linear transformation $g_Y^*$ of
$H^0(Y,\omega_Y^m[B_Y])$.
That is, for any $g\in\Bir(X_\eta\to\eta/k,B_{X_\eta})$,
there exists $g_Y\in \Bir(Y\to \eta/k,B_Y)$ such that
$$
g^*=(g_Y^*)\rest{H^0(X_\eta,\omega_{X_\eta}^m[mB_{X_\eta}])},
$$
where $g_Y^*$ is the representation of $g_Y$ on $H^0(Y,\omega_{Y/T}^m[mB_Y])$.

Now take $Y/\theta$ instead of $Y/\eta$,
where $Y\to \theta\to \eta$ is a Stein decomposition.
Then $Y$ is geometrically irreducible over $\theta$ and
$\Bir(Y\to\eta/k,B_Y)\subseteq \Bir(Y\to\theta/k,B_Y)$.
So, it is sufficient to establish the finiteness of
representation of $\Bir(Y\to\theta/k,B_Y)$ on
$H^0(Y,\omega_Y^m[B_Y])$.
But $(Y,B_Y)$ is klt by construction.

Step 3. Now $(X_\eta,B_{X_\eta})$ is klt and,
by Lemma~\ref{Burnside_3}, it is sufficient to verify that
linear $g^*$ is torsion for each $g\in\Bir(X_\eta\to\eta/k,B_{X_\eta})$.
Indeed, by the lemma
$$
g^*=g'^*g_\theta^*\text{ on } H^0(X_\eta,\omega_{X_\eta}^m[mB_{X_\eta}])=
H^0(\theta,\sM_\theta^m),
$$
where $g'\in\aBir(X_\eta\to\eta/k,B_{X_\eta}),
g_\theta\in\Aut(\theta/\overline{l},\sM_\theta)$,
$\sM_\theta=(\sM_\theta^m)^{1/m}$ is a canonical $\Q$-sheaf on $\theta$, and
$\sM_\theta^m$ is
the direct image of $\omega_{X_\eta}^m[mB_{X_\eta}]$ on $\theta$.
Since $B_{X_\eta}$ is a $\Q$-boundary, we can take
a canonical $\Q$-sheaf $\sM_\theta$.
The equation for sections under the direct image holds,
if $m$ is sufficiently divisible, e.g.,
$\sM_\theta^m$ is an invertible sheaf.
The last follows from above choice of $m$.
It is well-know that $g'^*$ is a bounded scalar torsion representation,
a representation on an isotrivial family
(see the proof of Corollary~\ref{repr_klt_lin} and Step 6 below).
Thus, for every torsion $g^*$, $g_\theta^*$ is also torsion.
By Lemma~\ref{Burnside_3} every $g_\theta$ and $g_\theta^*$ are
defined over $l_g$ with a uniformly bounded degree over
a field of finite type $l$ over the prime subfield in $k$.
Hence the representation $g_\theta^*$ satisfies (1) and (3)
of Theorem~\ref{Burnside} and is finite by the theorem.
This implies also the finiteness of $g^*$ because
the scalar part $g'^*$ is finite.
Actually, for sufficiently divisible $m$,
$g'^*$ is identical and $g^*=g_\theta^*$.

Step 4.
We can suppose now that $(X_\eta,B_{X_\eta})$ is klt,
equivalently, $\dk(X_\eta,B_{X_\eta})=\dim X_\eta$, and
$g$ is a generic flop.
We need to establish that $g^*$ is torsion for the linear representation.
In this step we verify semisimplicity of $g^*$, that is,
$g^*$ diagonalizable.
Moreover, $g^*$ is unitary:
the eigenvalues $e_i$ of $g^*$ have norm $1$.
It is sufficient to establish on a subspace of bounded forms
$$
W\subseteq \{\omega\in H^0(X_\eta,\omega_{X_\eta}^m[mB_{X_\eta}])\mid
\Vert\omega\Vert<+\infty\}.
$$

This is a birational concept: $\Vert\omega\Vert=\sup_{t\in T}\Vert\omega\Vert_t$,
the fiberwise norm.
A pedestrian and more algebraic explanation as follows.
For good properties
of $\Vert\omega\Vert_t$ on a completion of $\eta$, we
use a (flat) maximal wlc $(X/T,B)$ with tdlt singularities
such that $\eta$ is the generic point of $T$ and
$(X_\eta,B_{X_\eta})$ is as above.
Such a model exists.
We can suppose also that $B$ is horizontal over $T$, equivalently,
$B\sdiv=0$.
Then $\Bir(X_\eta\to\eta/k,B_{X_\eta})=\Bir(X\to T/k,B)$.
Usually, the induced morphism $g_T\colon T\dashrightarrow T$
is birational.
In particular, $t\mapsto t'=g_Tt$ and
fiberwise flops $g\rest{X_t}\colon (X_t,B_{X_t})\dashrightarrow (X_{t'},B_{X_{t'}})$
are not always well-defined.
They are defined for rather general points $t$ (and
so do powers $g^d$ for very general points).
The flop $g$ permutes some vertical b-divisors,
namely, multiple fibers and
degenerate fibers,
equivalently, the invariant divisors of log structure of $T$, over
generic points of which fibers are not reduced or
with degenerations (lc points).
This transformation on $X,T$ is really birational, that is,
some of those invariant divisors are contracted some are extracted
under $g,g_T$ respectively.
The moduli part of adjunction is stabilized over $X$:
$\sM=\overline{M}$, where $M$ is
an upper moduli part of adjunction for $(X/T,B)$, and
semiample by dimensional induction.
Moreover, under our assumptions $m\sM,mM$ are Cartier and
$mM$ is a divisor of the power sheaf $\omega_{X_\eta}^m[mB_{X_\eta}]= \omega_{X/T}^m[mB]$
of the sheaf of moduli part of adjunction.
Thus $H^0(X_\eta,\omega_{X_\eta}^m[mB_{X_\eta}])= H^0(X,\omega_{X/T}^m[mB])$
with isomorphic representations.

We denote by
$$
\varphi\colon X\to \PP(H^0(X_\eta,\omega_{X_\eta}^m[mB_{X_\eta}])^v)=
\PP(H^0(X,\omega_{X/T}^m[mB])^v)
$$
a contraction given by the linear system
$$
\linsys{\omega_{X_\eta}^m[mB_{X_\eta}]}=
\linsys{\omega_{X/T}^m[mB]}.
$$

Let $\Delta\subset T$ be the degeneration locus:
$$
\Delta=\{t\in T\mid
\dk(X_t,B_t)<\dim X_t=\dim X_\eta\}
$$
parameterizes the nonklt fibers.
By properties of norm,
$\Vert\omega\Vert_t$ is continuous always and bounded on $T$, if and only if
$\omega_t=0$, equivalently, $\Vert\omega\Vert_t=0$, for all $t\in\Delta$.
In the last situation $\Vert\omega\Vert=\max_{t\in T}\Vert\omega\Vert_t$.
So,
$$
W=H^0(X,{X_{\Delta,}}\red,\omega_{X/T}^m[mB])=\{\omega\in H^0(X,\omega_{X/T}^m[mB])\mid
\omega\rest{{X_{\Delta,}}\red}=0\}.
$$

By both definitions, $W$ is invariant under $g^*$.
The first definition uses the invariance of norm:
$\Vert g^*\omega\Vert=\Vert\omega\Vert$.
The second definition uses the invariance of degenerate fibers
for flops.
By properties  of norm (1) and (3)
the linear operators $(g^*)^n,n\in \Z$, are
uniformly bounded: for all integral numbers $n\in\Z$  and
all forms $\omega\in W$ of length $1$, $\Vert(g^*)^n\omega\Vert=\Vert\omega\Vert =1$.
Thus the operator $g^*$ is diagonalizable and unitary on $W$.

Now we establish the semisimplicity and unitary properties of
$g^*$ on the whole space $H^0(X_\eta,\omega_{X_\eta}^m[mB_{X_\eta}])=
H^0(X,\omega_{X/T}^m[mB])$.
Take for this a $g^*$-semiinvariant form $\omega_0\in W$ and
consider an equivariant imbedding of representation
(cf. the proof of Lemma~\ref{repres_comparable}): $g^*\omega_0=e_0\omega_0,e_0\in k^*$,
$$
H^0(X_\eta,\omega_{X_\eta}^m[mB_{X_\eta}])=
H^0(X,\omega_{X/T}^m[mB])\hookrightarrow
H^0(X,\omega_{X/T}^{2m}[2mB]), \omega\mapsto \omega\omega_0.
$$
Such a form $\omega_0$ exists for sufficiently large $m$ by
semiampleness of moduli part because $\varphi(X_{\Delta,}\red)$ is
a proper subset of $\varphi(X)$ [klt fibers are isotrivial families].
Actually, this restriction on $m$ was already imposed in Step 1:
the birational pre-image of $\sD$ on $\theta$ contains
all prime b-divisors over $\Delta$.

The image of the imbedding is a $g^*$-invarian subspace of bounded forms:
for any $t\in\Delta$ and any $\omega \in H^0(X,\omega_{X/T}^m[mB])$,
$$
g^*(\omega\omega_0)=e_0(g^*\omega)\omega_0 \text{ and }
(\omega\omega_0)_t=\omega_t\omega_{t,0}=\omega_t0=0.
$$
Thus $g^*$ is semisimple on $H^0(X_\eta,\omega_{X_\eta}^m[mB_{X_\eta}])$
with all $\vert e_i\vert=1$.

Step 5. Every $g^*$ is torsion on $W=H^0(X_\eta,\omega_{X_\eta}^m[mB_{X_\eta}])=
H^0(X,\omega_{X/T}^m[mB])$.
As one can see in the proof below,
we only need a completion along generic curves.
Again we use the regularization $(X/T,B)$.
According to Step 4 we need to establish that
each eigenvalue $e_i$ is a root of unity.
Let $w_i\in W$ be the eigenvectors of $g^*$.
By Step 4 , they generate $W$ and we can form a basis of those vectors
$w_1,\dots,w_d,d=\dim W$.
The dual vectors $w_i^v$ form a basis of
$H^0(X,\omega_{X/T}^m[mB])^v$.
Suppose that $w_1,\dots,w_l$ are all vectors $w_i$
with nonroots $e_1,\dots,e_l$.
We need to verify that $l=0$.

If $l\ge 1$, by Proposition~\ref{torus_invar-pt},
we can find a point $y\in \varphi(X)$ and
an integral number $n\not=0$ such that
\begin{description}

\item{\rm (1)}
$y$ has a nonzero coordinate $w_i(y),1\le i\le l$,
and

\item{\rm (2)}
$(g^*)^n$-invariant: $(g^*)^ny=y$.
\end{description}
Taking a power $(g^*)^n=(g^n)^*$ instead of $g^*$ we can suppose
$n=1$.
Note that the eigenvalues of $(g^*)^n$ are powers $e_i^n$ and
their property to be a root of unity independent of $n$.
By construction $\varphi(X)$ is invariant for $g^*$ and
nondegenerate.
Now we verify that $e_i$ is a root of unity, a contradiction.

Take now a point $t$ and a fiber $X_t$ over $y$.
The invariance $g^*y=y$ does not imply in
general invariance of $t$ and/or of $X_t$
under $g$.
But an invariance up to certain mp-trivial
deformation.
More precisely, $(X_t,B_{X_t}),t\in S$, belongs
to the family $(X_S/S,B_{X_S})$, where
$S\subseteq T$ is a maximal connected subvariety
such that $\varphi(X_S)=y$.
For every two points $t,s\in S$ and
sufficiently divisible even $m$ (as we assume, base point free $m\sM$),
there exists a canonical identification
of sections of their minimal lc centers:
$$
H^0(X_{S,}\red,\omega_{X_{S,}\red/S}^m[mB_{X_{S,}\red}])=H^0(Y_t,\omega_{Y_t}^m[mB_{Y_t}])=
H^0(Y_s,\omega_{Y_s}^m[mB_{Y_s}]),
X_{S,}\red=\varphi\1y,
$$
where $(Y_t,B_{Y_t}),(Y_s,B_{Y_s})$ are
minimal lc centers of $(X_t,B_{X_t}),(X_s,B_{X_s})$ respectively.
We denote this identification by
$c^*:H^0(Y_t,\omega_{Y_t}^m[mB_{Y_t}])\to
H^0(Y_s,\omega_{Y_s}^m[mB_{Y_s}])$.
It is determined by the relation: $c^*\rest{Y_t}=\rest{Y_s}$.
Actually, it is determined by the subfamily over $S$.
Apply Lemma~\ref{restr_in_mp_triv} to
the reduced \tdlt\ family $(X_{S,}\red/S,B_{X_{S,}\red})$.
However, it is not very useful, when $c^*$ does not
correspond to a flop (cf. the paragraph after the next one of this step).
In general, $Y_t,Y_s$ and, moreover,
$X_t,X_s$ even are not birationally equivalent.

On the other hand, by Proposition~\ref{flops_in_fibers},
for any $t\in S$ and any
minimal lc center $(Y_t,B_{Y_t})$,
there exists $s\in S$ such that, for
any minimal lc center $(Y_s,B_{Y_s})$,
there exists
a log flop
$g_{Y_t}\colon (Y_t,B_{Y_t})\dashrightarrow (Y_s,B_{Y_s})$.
Then under the above identification
$$
\rest{Y_t} g^*=(g_{Y_t})^*c^*\rest{Y_t}.
$$

We need to present now $c^*$ as a representation of
a (canonical) flop $c\colon (Y_s,B_{Y_s})\dashrightarrow (Y_t,B_{Y_t})$.
This is a log isomorphism and
this holds, e.g., if
there exist $Y_s,Y_t$ in the same isotrivial family
for minimal lc centers without degenerations.
The base $S$ can be present as a finite disjoint union
$\coprod S_i$ of locally closed subsets such that,
for every family $(X_{S_i,}\red/S_i,B_{X_{S_i,}\red})$, the family
of its minimal lc centers is finite disjoint union of
isotrivial families of klt $0$-pairs without degenerations.
The mp-trivial property of minimal lc centers follows by adjunction.
Since they are klt families, they are isotrivial by (Viehweg-Ambro).
For some curve $C$ (a curve $g^n(C)$) and some natural number $N>0$,
$g^N(C)$ gives a point $s$ in same $S_i$ as for $t$ and,
moreover, $Y_s$ is the same family as $Y_t$.
(Dirichlet principal.)
Now replace $g$ by $g^N$.
Then $s,t\in S_i$ and $Y_s,Y_t$ in the same klt
isotrivial family without degenerations.
So, $c^*$ correspond to a natural log isomorphism
$c\colon (Y_s,B_{Y_s})\to (Y_t,B_{Y_t})$.
(After a finite covering an isotrivial family without degeneration
became trivial.)

Now we take form the $\omega_i=w_i$.
Then $g^*\omega_i=e_i\omega_i$ and
$$
(cg_{Y_t})^*(\omega_i\rest{Y_t})=
(g_{Y_t})^*c^*(\omega_i\rest{Y_t})=
\rest{Y_t}(g^*\omega_i)=
\rest{Y_t}(e_i\omega_i)=e_i(\omega_i\rest{Y_t}).
$$
By construction $\omega_i\rest{Y_t}\not=0$,
equivalently, $\omega_i\rest{X_{S,}\red}\not=0$,
and
$cg_{Y_t}$ a flop of $(X_t,B_{X_t})$.
Hence $e_i$ is a root of unity by the next Step 6, a contradiction.

Step 6. $\dim T=0$ and $(X,B)$ is a klt $0$-pair
(cf. \cite[Theorem~3.9]{FC}).
Subtracting $B$ we can reduce the problem to that of in two situations

(1) with $B=0$ and $X$ is terminal,
and

(2) $(X,\varepsilon B),0<\varepsilon\ll 1$, is a klt Fano variety.

(Here we use induction on the dimension of fibers.)
The case (1) is well-known by \cite[Proposition~14.4]{U}:
every $e_i$ (actually single: $d=1$) is an algebraic integer and $\vert e_i\vert =1$.
So, $e_i$ is a root of unity.
In the case (2), the group $\Bir(X,B)$ is finite itself and
every $e_i$ is a root of unity ($1$-dimensional representation
of a finite group).
\end{proof}

The next result is a little bit more general (cf. Corollary~\ref{repr_lc_can})
but its proof uses more geometry: from isotrivial families to mp-trivial.

\begin{cor} \label{repr_wlc_can}
Let $(X_\eta,B_{X_\eta})$ be a generic wlc pair,
where $X_\eta$ is geometrically irreducible.
Then, for any natural number $m$,
the canonical representation
of generic log flops on differentials
$$
\Bir(X_\eta\to\eta/k,B_{X_\eta})\to
\Aut H^0(X_\eta,\omega_{X_\eta}^m[mB_{X_\eta}]), g \mapsto g^*,
$$
is finite.
Moreover, the order of representation
has a uniform bound, independent of $m$.
\end{cor}

\begin{proof}
Step 1. Construction of a generic lcm pair
$(X_{\eta,}\lcm,B_{X_{\eta,}\lcm})$.
The proof below uses the Iitaka contraction and the semiampleness conjecture
in the dimension of generic fiber.
However, it is possible to do without this assumption.
E.g., if a nonvanishing for generic fiber does not hold,
then $H^0(X_\eta,m\sM)=0$ for all natural $m$ and
the projective representation is empty.
The nonvanishing implies semiampleness by known results.
It is much easier for 2 section: $\dim H^0(X_\eta,m\sM)\ge 2$ (Kawamata).

Take a firberwise Iitaka contraction
$$
I\colon (X_\eta,B_{X_\eta})\to (X_{\eta,}\lcm,B_{X_{\eta,}\lcm}),
$$
where $(X_{\eta,}\lcm,B_{X_{\eta,}\lcm})$ is a generic lcm pair
with geometrically irreducible $X_{\eta,}\lcm$ and
with a boundary $B_{X_{\eta,}\lcm}$.
The boundary $B_{X_{\eta,}\lcm}$ is constructed by adjunction:
$B_{X_{\eta,}\lcm}=D+M$ is
a sum of the divisorial and a low moduli part of adjunction on $X_{\eta,}\lcm$.
The divisorial part of adjunction $D$ is determined canonically.

Step 2. Let
$$
G=\Bir(X_\eta\to\eta/k,B_{X_\eta})\cap
\ker \rho_\theta
\subseteq\Bir(X_\eta\to\eta/k,B_{X_\eta})
$$ be
a subgroup preserving any canonical upper and any effective low moduli part of
adjunction for $(X_\theta,B_{X_\theta})$, where
$\theta=X_{\eta,}\lcm, X_\theta=X_\eta,B_{X_\theta}=B_{X_\eta}$ and
$$
\rho_\theta\colon \Bir(X_\theta\to\theta/k,B_{X_\theta})\to
\Aut H^0(X_\theta,\omega_{X_\theta}^l[lB_{X_\theta}])
$$
for sufficiently divisible $l$.
The subgroup $G$ has a finite index in
$\Bir(X_\eta\to\eta/k,B_{X_\eta})$.
So, it is sufficient to establish the finiteness
of representation
$$
G\to
\Aut H^0(X_\eta,\omega_{X_\eta}^m[mB_{X_\eta}]), g \mapsto g^*.
$$

More precisely, we suppose that
$G$ preserves all differentials
$\omega\in H^0(X_\theta,\omega_{X_\theta}^l[lB_{X_\theta}])$:
for any natural number $l$ and any $g\in G$,
$g^*\omega=\omega$.
By Theorem~\ref{flop_repres}, the representation
$\rho_\theta$ is finite.
Indeed, by construction
$I\colon (X_\eta,B_{X_\eta})\to \theta$ is
a $0$-contraction and
$(X_\theta,B_{X_\theta})$ is a generic family of
$0$-pairs.

Note the $G$-invariance is an empty assumption
unless $B_{X_\theta}$ and $B_{X_\eta}$ are $\Q$-divisors
over $\theta$.
Indeed, otherwise, for all $l$,
$H^0(X_\eta,\omega_{X_\eta}^l[lB_{X_\eta}])=
H^0(X_\theta,\omega_{X_\theta}^l[lB_{X_\theta}])=0$, and
the corollary is established.
So, below we suppose that
$B_{X_\theta}$ and $B_{X_\eta}$ are $\Q$-divisors
over $\theta$, and
the bound on and kernel of $\rho_\theta$ are independent on $l$.
This is true for sufficiently divisible $l$.
Note also that each generic flop $g$ of
$(X_\eta,B_{X_\eta})$ is also a generic flop
of $(X_\theta,B_{X_\theta})$ and this gives
a natural inclusion:
$$
\Bir(X_\eta\to\eta/k,B_{X_\eta})\subseteq
\Bir(X_\theta\to\theta/k,B_{X_\theta}).
$$
Indeed, each fiberwise flop
$(X_t,B_{X_t})\dashrightarrow (X_{g_\eta t},B_{X_{g_\eta t}})$
is compatible with fibwerwise Iitaka contractions:
$$
\begin{array}{ccc}
(X_t,B_{X_t})&\dashrightarrow& (X_{g_\eta t},B_{X_{g_\eta t}})\\
\downarrow&&\downarrow\\
(X_{t,}\lcm,B_{X_{t,}\lcm})&\dashrightarrow&(X_{g_\eta t,}\lcm,B_{X_{g_\eta t,}\lcm})\\
\end{array}.
$$
So, the finiteness of $\ker \rho_\theta$ implies
the required finiteness of index.
The index has a uniform bound independent of $l$.

Step 3. For a rather divisible natural number $l$ and
any rather general effective $l$-canonical low moduli part of adjunction $M$,
there exists a canonical homomorphism of generic flops:
$$
\gamma=\gamma_M\colon G\to
\Bir(X_{\eta,}\lcm\to\eta/k,B_{X_{\eta,}\lcm}), g\mapsto g_{X_{\eta,}\lcm}.
$$
More precisely, the flop $g_{X_{\eta,}\lcm}$ is given fiberwise
by the above diagram:
$$
(X_{t,}\lcm,B_{X_{t,}\lcm})\dashrightarrow (X_{g_\eta t,}\lcm,B_{X_{g_\eta t,}\lcm}).
$$
Take such a natural number $l$ that
the upper effective $l$-canonical moduli part of adjunction
$lM\umod\in\linsys{\omega_{X_\theta}^l[lB_{X_\theta}]}$
on $X_\theta$ is mobile and b-free, that is,
the trace of a b-free divisor.
The moduli part is mobile even over $\eta$.
Then $I^*M=M\umod,g^*M\umod=M\umod$ and $g_{X_{\eta,}\lcm}^*M=M$.
The divisorial part of adjunction is preserved
by any generic flop of $(X_\eta\to\eta/k,B_{X_\eta})$ and
of $(X_{\eta,}\lcm\to\eta/k,B_{X_{\eta,}\lcm})$.

By Corollary~\ref{adjunt_for_0contr}, for rather general $M$,
$(X_{\eta,}\lcm,B_{X_{\eta,}\lcm})$ is
an lcm family, where $X_{\eta,}\lcm$ is
geometrically irreducible.
Since the boundary $B_{X_{\eta,}\lcm}$ depends on $M$,
for simplicity of notation, we replace it
by $B_{X_{\eta,}\lcm}+M$, where $B_{X_{\eta,}\lcm}$
denotes only the divisorial part of adjunction.
We use those notation in this proof.
Step 4.
Let
$$
G_\diamond=\{g\in G\mid
\gamma_Mg \text{ is almost identical }\}
\subseteq\Bir(X_\eta\to\eta/k,B_{X_\eta})
$$ be
a subgroup of $G$ which elements induces
almost identical flops of
$(X_{\eta,}\lcm\to\eta/k,B_{X_{\eta,}\lcm})$
for a rather general moduli part $M$.
The group $G_\diamond$ is independent of $M$ and
has a finite index in $G$.
(Another more invariant description of $G_\diamond$
see in the next step.)
So, it is sufficient to establish the finiteness
of representation
$$
G_\diamond\to
\Aut H^0(X_\eta,\omega_{X_\eta}^m[mB_{X_\eta}]), g \mapsto g^*.
$$

It is sufficient the finiteness of the quotients
$$
\Bir(X_{\eta,}\lcm\to\eta/k,B_{X_{\eta,}\lcm}+M)/
\aBir(X_{\eta,}\lcm\to\eta/k,B_{X_{\eta,}\lcm}+M)
$$
for rather general $M$.

Indeed, the group $\aBir(X_{\eta,}\lcm\to\eta/k,B_{X_{\eta,}\lcm}+M)$
acts on $X_{\eta,}\lcm$ within connected
isotrivial subfamilies of
the lcm  family $(X_{\eta,}\lcm,B_{X_{\eta,}\lcm}+M)$.
Such a group of automorphisms is finite
up to almost identical flops.
The quotient of $\aBir(X_{\eta,}\lcm\to\eta/k,B_{X_{\eta,}\lcm}+M)$
modulo almost identical flops has
a natural identification with a subgroup
of $\Aut(Y,B_Y)$, where $(Y,B_Y)$ is
an lcm pair canonically associated with
a rather general connected isotrivial subfamily
$(X_{S,}\lcm/S,B_{X_{S,}\lcm}+M\rest{X_{S,}\lcm})$ of
$(X_{\eta,}\lcm,B_{X_{\eta,}\lcm}+M)$.
For general $S$, $S$ is irreducible and the subfamily
reduced and geometrically irreducible over $S$.
Any generic flop
$g\in\aBir(X_{\eta,}\lcm\to\eta/k,B_{X_{\eta,}\lcm}+M)$
induces a flop
$$
g\rest{X_{S,}\lcm}\in
\Bir(X_{S,}\lcm\to S/k,B_{X_{S,}\lcm}+M\rest{X_{S,}\lcm})=
\aBir(X_{S,}\lcm\to S/k,B_{X_{S,}\lcm}+M\rest{X_{S,}\lcm}).
$$
By Lemma~\ref{mp_it} the family over $S$ is mp-trivial and
by definition, there exists a natural contraction
(we can suppose $S$ to be complete)
$$
\varphi\colon(X_{S,}\lcm/S,B_{X_{S,}\lcm}+M\rest{X_{S,}\lcm})\to Y,
$$
given by the moduli part of adjunction, that is,
$Y$ is projective with a polarization $H$ such that
$\varphi^*H$ is an upper moduli part of adjunction.
So, any generic flop $g$ of
$(X_{S,}\lcm/S,B_{X_{S,}\lcm}+M\rest{X_{S,}\lcm})$
induces a regular automorphism of $Y$ (linear
for very ample $H$).
For the lcm family there exists a natural boundary
$B_Y$ on $Y$ such that
$$
\varphi\rest{X_{t}\lcm}\colon (X_{t,}\lcm,B_{X_{t,}\lcm})\to
(Y,B_Y)
$$
is a fliz and $g\rest{X_{S,}\lcm}$ induces
a flop $g_Y$ of $(Y,B_Y)$.
As above $B_Y$ depend on $M$ and can be replaced by
$B_Y+M_Y$.
The almost identical flops $g\rest{X_{S,}\lcm}$ induces
identical automorphism on $(Y,B_Y)$.
A fiberwise flop
$$
g\rest{X_{t,}\lcm}\colon
(X_{t,}\lcm,B_{X_{t,}\lcm})\to
(X_{g_St,}\lcm,B_{X_{g_St,}\lcm}),
t\in S,
$$
of almost
identical flop
is canonical, that, correspond to identical
on a trivialization of the family.
So, to be almost identical is generic
deformation property for deformation of an isotrivial family.
The group $\Aut(Y,B_Y+M)$ is finite and
also a generic deformation invariant.
Thus the almost isotrivial flops of $(X_{\eta,}\lcm,B_{X_{\eta,}\lcm}+M)$
form a finite group up to almost identical ones.
Moreover, there exists a uniform bound on
the flops of isotrivial subfamilies up to
almost identical flops.

On the other hand,
the group of generic flops modulo almost isotrivial flops
is finite, because the family
$(X_{\eta,}\lcm,B_{X_{\eta,}\lcm}+M)$ is lcm.
The bound for the quotient can be given by
the degree of $\theta\to\fM$,
where $\fM$ is a coarse moduli for fibers and
$T\to\theta\to\fM$ is a Stein decomposition of
the moduli morphism.

So, each subgroup
$$
\{g\in G\mid
\gamma_Mg \text{ is almost identical }\}
\subseteq\Bir(X_\eta\to\eta/k,B_{X_\eta})
$$
has a finite index for every $M$.
Actually, the group is independent of $M$,
because the isotriviality of subfamily
over $S$ and the contraction $\varphi$
are independent of $M$.
Indeed, if $M'$ is another (generic) effective moduli
part $M\sim_l M'\ge 0$, then
$M'$ is also vertical with respect to $\varphi$ and
$0\le M_Y'\sim_l M_Y$.

Step 5. The projective representation
$$
G_\diamond\to
\Aut \PP(H^0(X_\eta,\omega_{X_\eta}^m[mB_{X_\eta}])), g \mapsto g^*,
$$
is trivial.
It sufficient to verify that, for any flop $g\in G_\diamond$ and
any effective divisor $D\in\linsys{\omega_{X_\eta}^m[mB_{X_\eta}]}$,
$$
g^*D=D.
$$
If $D$ is fixed, it is sufficient to verify the same
property on a rather general mp-trivial subfamily,
which is invariant for $g$.
Take a subfamily $(X_S/S,B_{X_S})$ over
a generic isotrivial family
$(X_{S,}\lcm/S,B_{X_{S,}\lcm}+M\rest{X_{S,}\lcm})$ of
Step 4.
The letter family is mp-trivial and
so does the former one.
Moreover, the effective moduli parts are the same
under the Iitaka contraction $I$:
$$
D\rest{X_S}=I\rest{X_S}^*D\lcm=I^*\varphi^*D_Y,
$$
where $D\lcm\in\linsys{\omega_{X_{S,}\lcm}^m[mB_{X_{S,}\lcm}]}$ and
$K_Y+B_Y\sim_\R D_Y\ge 0$ is an effective divisor on $Y$.
Hence
$$
g^*D\rest{X_S}=I^*g_{X_{S,}\lcm}^*D\lcm=I^*\varphi^*g_Y^*D_Y=
I^*\varphi^*D_Y=D\rest{X_S},
$$
because $g_Y=\id_Y$.

Step 6.
The scaler representation
$$
G_\diamond\to
\Aut H^0(X_\eta,\omega_{X_\eta}^m[mB_{X_\eta}]), g \mapsto g^*,
$$
is finite with a uniform bound.
It is scaler by Step 5.
So, for every $g\in G_\diamond$,
there exists a constant $e\in k^*$ such that,
for every $\omega\in H^0(X_\eta,\omega_{X_\eta}^m[mB_{X_\eta}])$,
$$
g^*\omega=e\omega.
$$

Take a rather general mp-trivial family $(X_S/S,B_{X_S})$ of Step 5.
Then, for general
$\omega \in H^0(X_\eta,\omega_{X_\eta}^m[mB_{X_\eta}])$,
$\omega\rest{X_S}\not=0$.
For general $t\in S$, $g_St=s\in S$, and
there exists a flop
$$
g\rest{X_t}\colon (X_t,B_{X_t})\dashrightarrow (X_s,B_{X_s})
$$
and a canonical log isomorphism with respect to restrictions
$$
c\colon (X_t,B_{X_t})\dashrightarrow (X_s,B_{X_s}).
$$
Then $g_t=c\1 g\rest{X_t}$ is a flop of $(X_t,B_t)$ and,
for even $m$ and general $\omega\in H^0(X_\eta,\omega_{X_\eta}^m[mB_{X_\eta}])$,

$$
g_t^*(\omega\rest{X_t})=g\rest{X_t}^*{c\1}^*(\omega\rest{X_t})=
g\rest{X_t}^*(\omega\rest{X_s})=(g^*\omega)\rest{X_t}=
e\omega\rest{X_t}\text{ and } \omega\rest{X_t}\not=0.
$$
For any $m$, we can consider $\omega^2$ and $e^2$.
So, $e$ is a root of unity.
There are only finitely many such roots.
The number of roots depends only on $(X_t,B_{X_t})$.
Using the Iitaka contraction $I\rest{X_t}$, one
can reduce the scaler representation to
a fiber of $I\rest{X_t}$, that is, to a fixed $0$-pair
(cf. Step 6 in the proof of Theorem~\ref{flop_repres}).
\end{proof}

\begin{thm}\label{repr_slc_proj}
Let $(X_\eta,\sB_{X_\eta})$ be a generic normally lc [slc] pair
with a b-boundary $\sB_{X_\eta}$,
$G\subseteq \Bir(X_\eta\to\eta/k,\sB_{X_\eta})$ be a subgroup of generic flops,
and
$\sD$ be a b-divisor of $X_\eta$ in
a decomposition
$$
r(\sK_{X_\eta}+\sB_{X_\eta})\equiv \sF+\sD,
$$
where
\begin{description}

\item{\rm (1)}
$r$ is nonnegative real number,

\item{\rm (2)}
$\sF$ is an effective b-divisor, invariant for $G$, and

\item{\rm (3)}
$\sD$ is an effective b-divisor, invariant
up to linear equivalence for $G$.

\end{description}
Then, for any natural number $m$, the projective (sub)representation
of generic log flops
$$
G\to \Aut \PP(H^0(X_\eta,m\sD)), g \mapsto g^*,
$$
is finite.
Moreover, the order of representation has
a uniform bound, independent of $m,r,\sF,\sD,G$.
\end{thm}

\begin{proof}
Step 1. We can suppose that $X_\eta$ is normal, irreducible and
geometrically irreducible.
Take a normalization $(X_\eta\nor,B_{X_\eta\nor})$ and
its irreducible decomposition $(X_\eta\nor,B_{X_\eta\nor})=
\coprod (X_i,B_{X_i})$.
Then by definition the normal pair $(X_\eta\nor,B_{X_\eta\nor})$
is lc and the decomposition of
log canonical divisor is componentwise:
$r(\sK_{X_i}+\sB_{X_i})\equiv \sF_i+\sD_i$.
But generic flops permute components, that is,
a canonical homomorphism
$$
G\subseteq \Bir(X_\eta\to\eta/k,B_{X_\eta})\to \Aut\{X_i\},
g\mapsto (X_i\mapsto g(X_i)),
$$
is defined.
On the other hand, the representation
$$
G\to \Aut \PP(H^0(X_\eta,m\sD))=
\prod\Aut\PP(H^0(X_i,m\sD_i))
$$
agrees with permutations.
The group of permutations is finite and
the representation of kernel is
in a product of restricted representations
$$
\ker[G\to \Aut\{X_i\}]\subseteq
\prod G_i\to\prod\Aut\PP(H^0(X_i,m\sD_i)),
$$
where $G_i=(\ker[G\to \Aut\{X_i\}])\rest{X_i}\subseteq\Bir(X_i\to\eta/k,B_{X_i})$.
Thus it is sufficient to verify the finiteness
of each factor
$G_i\to\Aut\PP(H^0(X_i,m\sD_i))$.
This means that we can assume that $X_\eta$ is
normal and irreducible.
A finite base change (Stein decomposition)
allows to assume geometrical irreducibility of $X_\eta$.
This change can increase the group of generic flops and its subgroup $G$, but
preserves decomposition and sections.
Indeed, $K_\eta=K_\theta$ for any decomposition
$X_\eta\to\theta\to\eta$, where $\theta\to\eta$ is finite.
Thus we can take the same decomposition
$r(\sK_{X_\theta}+\sB_{X_\theta})\equiv \sF+\sD$.
Actually, we can replace $G$ by a larger subgroup:
the generic flops $g$, which preserve $\sF$ and
preserve up to linear equivalence $\sD$.
Anyway, this subgroup includes the flops
from $G$ by (2-3).

Step 2.
Since the divisor $\sF+\sD$ is invariant up to
linear equivalence, we suppose also that $\sF=0$.
By Lemma~\ref{prod_of_repres}, (1),
the representation on the subspace
$$
\PP(H(X_\eta,m\sD))\subseteq \PP(H^0(X_\eta,m(\sF+\sD)),
D\mapsto D+\sF,
$$
is invariant and finite, if
the representation is finite on the ambient space.
Indeed, the lemma applies, if $H^0(X,m\sD)\not=0$.
Otherwise, the representation on $H^0(X,m\sD)$ is
empty.

Step 3. We can suppose that $(X_\eta,B_{X_\eta})$ is
wlc.
Indeed, by our assumptions, it is an initial model.
Hence we can apply the LMMP.
If the resulting model $(X_\eta/\theta,B_{X_\eta})$
is a Mori fibration, then b-divisors $\sK_{X_\eta}+\sB_{X_\eta}$ and
$\sD$ are negative with respect to the fibration and,
for $r>0$.
$H^0(X_\eta,m\sD)=0$ and the representation is empty.
Otherwise, in the Mori case, $r=0$ by (1) and
$\sD\equiv 0$.
So, the representation is empty or trivial,
respectively, for $H^0(X_\eta,m\sD)=0$ or $=k$.

Therefore, the nontrivial cases are possible only
for a wlc resulting model.
Note that the sections, the numerical equivalence
and
representation will be preserved under the LMMP
modifications.
(Even the subgroup $G$ of generic flops under (2-3) can be increased.)

Step 4.
Finally, we derive the required finiteness from
Corollary~\ref{repr_wlc_can} and
Lemma~\ref{repres_comparable}.
By Corollary~\ref{repr_wlc_can}, the (sub)representation of
$G\subseteq \Bir(X_\eta\to\eta/k,B_{X_\eta})$ on
$H^0(X_\eta,\omega_{X_\eta}^l[lB_{X_\eta}])$
is finite for any natural number $l$.
The projective representations of
$G\subseteq \Bir(X_\eta\to\eta/k,B_{X_\eta})$ on
$\PP(H^0(X_\eta,\omega_{X_\eta}^l[lB_{X_\eta}]))$ and
on $\PP(H^0(X,l(\sK_{X_\eta}+\sB_{X_\eta})))$
are canonically isomorphic and finite.
Thus by Lemma~\ref{repres_comparable} the representation
on $H^0(X,m\sD)$ is finite too.
A uniforme bound can be found by
Corollary~\ref{repr_wlc_can}.
\end{proof}

\begin{cor} \label{repr_lc_can}
Let $(X_\eta,B_{X_\eta})$ be a generic slc
pair with a boundary $B_{X_\eta}$.
Then, for any natural number $m$,
the canonical representation
of generic log flops on differentials
$$
\Bir(X_\eta\to\eta/k,B_{X_\eta})\to
\Aut H^0(X_\eta,\omega_{X_\eta}^m[mB_{X_\eta}]), g \mapsto g^*,
$$
is finite.
Moreover, the order of representation
has a uniform bound, independent of $m$.
\end{cor}

Actually, the slc property of the statement can
be replaced by many similar ones, e.g.,
normally lc, seminormal lc, etc.
Then a proof should only explain
what is a meaning of differentials or of
$H^0(X_\eta,\omega_{X_\eta}^m[mB_{X_\eta}])$ and of flops.
If this is natural, then a proof goes as
below in the slc case.
For instance, generic flops should preserve
such differentials for every $m$.

\begin{proof}
This proof uses the reduction to
geometrically wlc pairs, which can be
done as in Theorem~\ref{repr_slc_proj}.
After that for the linear representation
we can apply Corollary~\ref{repr_wlc_can}.

Take a normalization $(X_\eta\nor,B_{X_\eta\nor})$ and
its irreducible decomposition $(X_\eta\nor,B_{X_\eta\nor})=
\coprod (X_i,B_{X_i})$.
Then by definition there exists a natural imbedding
$$
H^0(X_\eta,\omega_{X_\eta}^m[mB_{X_\eta}])\hookrightarrow
H^0(X_\eta,\omega_{X_\eta}^m[m\sB_{X_\eta}])=
H^0(X_\eta\nor,\omega_{X_\eta\nor}^m[mB_{X_\eta\nor}])=
$$
$$
\prod H^0(X_i,\omega_{X_i}^m[mB_{X_i}]),
$$
where the differentials for the b-boundary
$\sB_{X_\eta}$ are defined on the normalization.
This an imbedding of linear representation too.
Thus by Theorem~\ref{repr_slc_proj}
the projective representation
$$
\Bir(X_\eta\to\eta/k,B_{X_\eta})\to
\Aut \PP(H^0(X_\eta,\omega_{X_\eta}^m[mB_{X_\eta}])), g \mapsto g^*,
$$
is uniformly finite.

Actually, the big linear representation on the product is
also uniformly finite.
It is sufficient to verify for an irreducible component and
wlc by the MMMP.
The required finiteness in this case follows from
Corollary~\ref{repr_wlc_can}.
\end{proof}

\begin{cor}\label{repr_lc_proj}
Let $(X_\eta,B_{X_\eta})$ be a generic slc
pair with a boundary $B_{X_\eta}$,
and $\sM$ be an upper maximal (canonical) moduli part of adjunction.
Then, for any natural number $m$, the projective representation
of generic log flops
$$
\Bir(X_\eta\to\eta/k,B_{X_\eta})\to \Aut \PP(H^0(X_\eta,m\sM)), g \mapsto g^*,
$$
is finite.
Moreover, the order of representation has
a uniform bound, independent of $m,\sM$.
\end{cor}

\begin{proof}
Immediate by Theorem~\ref{repr_slc_proj}.
By definition a b-divisor $\sM$ is a mobile part of
a mobile decomposition:
$$
\sK_{X_\eta}+\sB_{X_\eta}\sim \sF+\sM.
$$
Then we use the invariance of $\sF$ and
invariance up to linear equivalence of $\sK_{X_\eta}+\sB_{X_\eta}$ and $\sM$
with respect to generic flops.
\end{proof}

\begin{cor}\label{repr_klt_lin}
Let $(X_\eta,B_{X_\eta})$ be a generic klt
pair with a boundary $B_{X_\eta}$,
$G\subseteq \Bir(X_\eta\to\eta/k,B_{X_\eta})$
be a subgroup of generic flops,
and
$\sD$ be a b-divisor of $X_\eta$
as in Theorem~\ref{repr_slc_proj}.
In addition, we assume either $\sD$ is $G$-invariant, or
it is a $G$-invariant b-divisorial sheaf.
Then, for any natural number $m$, the linear representation
of generic log flops
$$
G\to \Aut H^0(X_\eta,m\sD), g \mapsto g^*,
$$
is finite.
Moreover, the order of representation has
a uniform bound, independent of $m,G$.
\end{cor}

The bound on order can depend on $r,\sF,\sD$.

\begin{proof}
Immediate by Theorem~\ref{repr_slc_proj} and the finiteness of
scaler representations in the klt case.

We can suppose that $H^0(X_\eta,m\sD)$ is
not empty for some $m\ge 1$.
Otherwise all representations are empty.
Taking such a minimal natural $m$ and
replacing $\sD$ by $m\sD$ (respectively,
$r$ by $mr$ etc), we suppose
that $H^0(X_\eta,\sD)\not=0$.
So, there exists a nonzero rational function $F\in k(X_\eta)$
such that $F\in H^0(X_\eta,\sD)$.

By Theorem~\ref{repr_slc_proj}
the subgroup
$$
G_\diamond=\{g\in G\mid \text{ for all } m,
g^* \text{ is identical on }
\PP(H^0(X_\eta,m\sD))\}\subseteq G
$$
of the scaler representation
has a finite index, uniformly bounded with respect to $m$.
So, $g^*F=c_gF,c_g\in k^*$, for all $g\in G_\diamond$ and
it is sufficient to establish the finiteness
for the scaler representation.
The scaler representation
$$
G_\diamond\mapsto k^*,g\mapsto g^*=c_g,
$$
depends on $F$ and is finite, that is, $c_g$
belongs to a finite set of roots of unity.
This implies the required finiteness of
linear representation uniformly for all $m$.

The finiteness of the scaler representation of
$G_\diamond$ on $F\sO_X=\sO((F))$ follows
from the klt property of $(X_\eta,B_{X_\eta})$.
The question can be reduced to situation
with the scaler representation for a klt $0$-pair
$(X,B)$.
Moreover, this follows from the finiteness of
the linear canonical representations of
$(X,B+\varepsilon\Supp(F))$,
where $0<\varepsilon\ll 1$ is a small positive (rational) real number.

The similar approach works for the b-divisorial
sheaves $\sO_X(\sD)$.
\end{proof}

\begin{cor}\label{generic_Kawamata}
Let $(X_\eta,B_{X_\eta})$ be a generic wlc
pair.
Then there are only finitely many generic
log flops of $(X_\eta\to \eta/k,B_{X_\eta})$
up to mp-autoflops, with respect to a maximal moduli part of adjunction, that is,
the group
$$
\Bir(X_\eta\to\eta/k,B_{X_\eta})/\mBir(X_\eta\to\eta/k,B_{X_\eta})
$$
is finite.
\end{cor}

\begin{proof}

Step 1.
After a finite base change (extension) we can suppose that $X_\eta$ is
geometrically irreducible.
For a base change, the group of generic flops increases, but
the group of mp-autoflops decreases.

Step 2. After an appropriate perturbation we can suppose
that $B_{X_\eta}$ is a $\Q$-divisor and $\sM$ is a $\Q$-divisor too.
By definition the b-divisor $\sM$ is a moduli part of adjunction
for a maximal model.
It exists.
The moduli part of adjunction is invariant of generic flops:
$g^*\sM\sim_m \sM$ for any rather divisible natural number $m$.

Now, for such a number $m$,
take the canonical semirepresentation of generic log flops
$$
\Bir(X_\eta\to \eta/k,B_{X_\eta})\to \Aut H^0(X_\eta,m\sM),g\mapsto g^*.
$$
A posteriori we can convert it into a noncanonical representation.

Step 3. The kernel of representation is $\mBir(X_\eta\to \eta/k,B_{X_\eta})$
for any rather divisible natural number $m$.
Indeed, consider the morphism
$$
\varphi\colon X_\eta\to \PP(H^0(X_\eta,m\sM)^v)
$$
given by the linear system $\linsys{m\sM}$.
The above representation gives a canonical representation on the projectivisation:
$$
\Bir(X_\eta\to \eta/k,B_{X_\eta})\to \Aut \PP(H^0(X_\eta,m\sM)^v),g\mapsto g^*.
$$
The rational morphism $\varphi$ is equivariant with respect
to the action of generic flops, and,
 for any rather divisible $m$, is actually
a morphism and a contraction.
The kernel of representation can be determined on
(finitely many) rather general fibers $\varphi\1 x\in\varphi(X_\eta)$.
Those fibers are irreducible and the kernel acts within them.
On the other hand, $\sM\rest{\varphi\1 x}\sim_m 0$ and
by adjunction the restriction is also a maximal moduli part of adjunction
on the subfamily for $\varphi\1 x$.
Thus the kernel consists of mp-autoflops.
The converse holds as well.

So, for such a natural number $m$, the image of
projective representation is isomorphic to
the quotient group in the statement.

Finally, the image is finite by Corollary~\ref{repr_lc_proj}.
\end{proof}

\section{Bounding flops}

\begin{con}[Kawamata~{\cite[Conjecture~3.16]{ISh}}]
\label{Kawamata_conjecture} The number of projective klt wlc
models in a given log birational class is always finite up to log
isomorphisms.
\end{con}

\begin{ex}[Pjateckii-Shapiro and Shafarevich]
Let $X$ be a nonsingular K3 surface.
Conjecture~\ref{Kawamata_conjecture} holds for $X$.
That is, $X$ has finitely many wlc klt models $Y$
up to isomorphism.
Actually, each model $Y$ is a $0$-pair with $B=0$ and
only Du Val singularities.
The polarized lattices $\Lambda^+Y\subset\Lambda(Y)$
of models $Y$ have finitely many types.
This implies that the models have bounded polarization.

The same holds for (genus $1$) fibrations $Y\to T$.
There are only finitely many fibrations up to isomorphism,
where $Y$ is a wlc klt model of $X$.

In terms of $\Aut(X)$ these facts means that
there are finitely many orbits of exceptional curves
(not necessarily irreducible) and finitely many
orbites of fibrations.
All these follows the Torelli theorem
for K3 surfaces \cite{PShSh}.

So, the group of automorphisms $\Aut(X)$ is
infinite if $X$ has infinitely many exceptional curves or/and
fibrations.
The converse does not hold in general.
\end{ex}

Let $(X,B)$ be a pair with a boundary $B$.
Denote by $\GM(X,B)$ the category of
projective klt wlc models $(Y,B\Lg_Y)$ of $(X,B)$ with
their log flops $(Y,B\Lg_{Y})\dashrightarrow (Y',B\Lg_{Y'})$
as morphisms which are considered up to log isomorphisms.
For example, if $A$ is an Abelian variety and $B=0$
then $\GM(A,0)$ is equivalent to a trivial one, a category
with a single object $A$ and with only the identical morphism.

\begin{df}[Bounded flops]
Let $d$ be a natural number.
A log flop of projective klt wlc models
$(X_1,B_{X_1})\dashrightarrow (X_2,B_{X_2})$ is
{\em bounded with respect to\/} $d$, if there are very ample divisors
$D_1,D_2$ on $X_1,X_2$ respectively of degree $\le d$ and
of mutual degrees $\le d$.
A {\em set\/} (or {\em class\/}) {\em of log flops\/}
is {\em bounded\/}, if there exists a natural number $d$ with
respect to which the flops are bounded.
A {\em category of log flops\/} is {\em of bounded type\/},
if the category has a bounded {\em set\/} (or {\em class\/})
of generators.

A {\em model\/} $(X,B)$ is bounded with respect to $d$, if the
identical flop $(X,B)\to (X,B),x\mapsto x$, so does. A {\em set\/}
(or {\em class\/}) {\em of pairs\/} $(X,B)$ is bounded, if there
exists a natural number $d$ with respect to which the pairs are
bounded.
\end{df}

We denote by $\GM\bnd(X,B)\subseteq \GM(X,B)$
a subcategory of bounded type of log flops up to log isomorphisms.
For example, for any pair $(Y,B\Lg_{Y})$ in $\GM(X,B)$,
the subcategory of log isomorphisms
$(Y_1,B_{Y_1})\to (Y_2,B_{Y_2})$, where
$(Y_1,B_{Y_1}),(Y_2,B_{Y_2})$ are log isomorphic to
$(Y,B\Lg_{Y})$, is of bounded type.

According to Corollary~\ref{equiv-conjectures} below
Conjecture~\ref{Kawamata_conjecture}
is equivalent to each of the following one.

\begin{con} \label{latt_fg_conjecture}
The category $\Lat(X,B)$ is of finite type.
\end{con}

\begin{con} \label{boundedness_conjecture}
The models of $\GM(X,B)$ are bounded.
\end{con}

\begin{con} \label{categ_bound_conjecture}
The category $\GM(X,B)$ is of bounded type.
\end{con}

Note that, in general, $(X,B)$ may not have
a projective klt wlc model at all.
Then the conjectures are empty.
However, if $(X,B)$ has a projective klt wlc model
then any other resulting projective model $(Y,B\Lg_Y)$
is also klt wlc.

\begin{thm} \label{weak_Kawamata}
Any category of bounded type $\GM\bnd(X,B)$ is of finite type.
\end{thm}

\begin{cor}\label{equiv-conjectures}
Conjectures~\ref{Kawamata_conjecture}, ~\ref{latt_fg_conjecture},
\ref{boundedness_conjecture} and \ref{categ_bound_conjecture}
are equivalent.
\end{cor}

\begin{proof}
Conjecture~\ref{Kawamata_conjecture} implies Conjecture~\ref{latt_fg_conjecture}.
Consider a subcategory of bounded type $\GM\bnd\subseteq\GM(X,B)$
with finitely many objects $(Y,B\Lg_Y)$ such that
each wlc model of $(X,B)$ is isomorphic to one of in the subcategory.
Thus the subcategory is equivalent to the whole one.
Actually, the subcategory is of finite type.
The generators are projective $\Q$-factorializations,
elementary contractions and flops.
There are only finitely many such transformations.
Up to log isomorphisms,
they belong to $\GM\bnd$ and every flop of $\GM\bnd$
can be factorize into them \cite{ShC}.
Hence the image of the lattice functor
$$
\GM(X,B)\to\Lat(X,B), (Y,B\Lg_Y)\mapsto \Lambda(Y),
$$
is  of finite type too, Conjecture~\ref{latt_fg_conjecture}.
Indeed, the image is equivalent to the image of
$\GM\bnd$.

Conjecture~\ref{latt_fg_conjecture} implies Conjecture~\ref{boundedness_conjecture}.
The former implies that there are finitely many types
of polarized lattices $\Lambda^+\subset\Lambda$ for models
$(Y,B\Lg_Y)$ of $\GM(X,B)$.
For every polarization type, take a polarization $H\in\Lambda^+$.
So, each model in $\GM(X,B)$ has a bounded polarization
by the effective ampleness: there exists a natural number $N$
such that $NH$ is very ample for every $(Y,B\Lg_Y)$ of
type $H\in\Lambda^+\subset\Lambda$.
Hence each model of $\GM(X,B)$ is bounded, Conjecture~\ref{boundedness_conjecture}.

Conjecture~\ref{boundedness_conjecture} implies Conjecture~\ref{categ_bound_conjecture}.
The former implies that the objectes are bounded,
that is, each model $(Y,B\Lg_Y)$ has a bounded polarization $H_Y$.
Equivalently, the models belong to pairs of finitely many families
of triples.
By \cite{ShC} projective $\Q$-factorializations,
elementary contractions and flops are
generators of the generalized log flops.
Those generators are bounded by Noetherian induction for above families.
Indeed, a relative $\Q$-factorialization can be done for klt
families generically.
So, the $\Q$-factorializations are bounded.
Each elementary contraction $(Y_1,B\Lg_{Y_1})\to (Y_2,B\Lg_{Y_2})$
can be treated as a crepant elementary blowup of
an exceptional divisor $E\subset Y_1$ with $b_{E}=\mult_EB$.
There are only finitely many such exceptional b-divisors
for $Y_2$.
Again, by Noetherian induction,
the blowups form finitely many projective families and are bounded.
Each elementary flop $(Y_1,B\Lg_{Y_1})\dashrightarrow (Y_2,B\Lg_{Y_2})$
can be factorize into an elementary flopping contraction
$(Y_1,B\Lg_{Y_1})\to (Y,B\Lg_Y)$ and
a small blowup $(Y,B\Lg_Y)\leftarrow (Y_2,B\Lg_{Y_2})$.
Both are projective $\Q$-factorializations with two possible polarizations.
So, their composition is also bounded.

Conjecture~\ref{categ_bound_conjecture} implies Conjecture~\ref{Kawamata_conjecture}.
By the former conjecture we can take $\GM\bnd(X,B)=\GM(X,B)$.
Then by Theorem~\ref{weak_Kawamata} the category $\GM(X,B)$ is
of finite type.
In particular, $\GM(X,B)$ is equivalent to a category with
finitely many objects, Conjecture~\ref{Kawamata_conjecture}.
\end{proof}

\begin{proof}[Proof of Theorem~\ref{weak_Kawamata}]
Consider a bounded category $\GM\bnd=\GM\bnd(X,B)$ of
klt wlc models $(Y,B\Lg_Y)$ of a pair $(X,B)$.
Since the category is bounded,
there exists a bounded coarse muduli $\fM$ of triples
$(X,B,H)$, where now $(X,B)$ denotes a klt wlc model with
a polarization $H$, such that the bounded models of
$\GM\bnd$ belong to $\fM$.

Step 1. We can suppose that the models of $\GM\bnd$ are
Zariski dense in $\fM$.
This means that triples $(X,B,H)$ with a pair $(X,B)$ in $\GM\bnd$
form a dense subset in $\fM$.
Otherwise, we replace $\fM$ by a Zariski closure of
those triples.
A polarization $H$ is considered here
as an invertible sheaf up to algebraic equivalence,
that is, $H\in \NS X=\Pic X/\approx=\Pic X/\Pic_0 X$.
The corresponding b-sheaf modulo $\approx$ will
be denoted by $\sH\in\bNS X$.

We assume also that the moduli is irreducible
because it is sufficient to establish
the theorem for pairs of each irreducible component.
So, there is a bounded reduced irreducible family
$(X/T,B,H)$ of $\fM$ such
that it contains up to a log isomorphism
a dense subset of pairs of $\GM\bnd$.
That is, for such a pair $(X_t,B_{X_t})$ there exists
a polarization $H_t$ on $X_t$
such that $(X_t,B_{X_t},H_t)$ belongs to $(X/T,B,H)$.

Step 2. There is such a family $(X/T,B,H)$ with
a finite set of
b-polarizations $\sD_i\subset X$ over $T$
such that each
flop of $\GM\bnd$ can be given by some of those divisors.
This means that, if $t,s\in T$ and
$g_t\colon (X_t,B_{X_t})\dashrightarrow (X_s,B_{X_s})$ is
flops of $\GM\bnd$,
then, for some b-divisor $\sD_i$,
the flop is given (as directed) for $\sD_{t,i}$.
More precisely, we suppose that $g(\sD_{t,i})= \sH_s$.
By the boundedness of flops, the restriction $\sD_{t,i}$ is bounded with
respect to $H_t$.
Each $\sD_i$ is defined over
a locally closed algebraic subvariety $T_i\subseteq T$.
By the irreducibility of $T$ and the dense property of Step 1,
at least one $T_i$ is dense: $\overline{T_i}=T$,
equivalently, $\sD_i$ is dominant over $T$.
By Noetherian induction,
it is sufficient to verify the finiteness
of $\GM\bnd$ for flops given by the dominant $\sD_i$.
Since the set of b-divisors $\sD_i$ is finite,
we can suppose that
each $\sD_i$ is flat over $T$ and surjective to $\fM$.
By construction each $\sD_i$ is a b-polarization over $T$,
that is, it is big and semiample over $T$.

Each
b-polarization $\sD\in\bNS X/T$ gives a canonical log flop
$c=c_\sD \colon (X/T,B,H)\dashrightarrow (X'/T,B_{X'},H')$
with $\sH'= c_*\sD$, equivalently,
$\sD=\overline{H'}$ as b-divisors.
In general, the second family does not belong to $\fM$.
Moreover, that can happen with flops for $\sD_i$.
However, there exists the image of $(X'/T,B_{X'},H')$
in $\fM$: the image of subfamily over
$$
T'=\{t\in T\mid (X_t',B_{X_t'},H_ t')\in \fM\}\subseteq T.
$$
Again by the dense property we can suppose
that $T_i=T'$ for some $\sD=\sD_i$ is dense in $T$.
If $\sD$ is flat over $T$ then
the dense property implies the equality: $T'=T$.
By Noetherian induction we can
suppose that, for each $\sD_i$,
$T_i$ is dense in $T$ and, actually, each $T_i=T$.
Thus each $(X_i/T,B_{X_i},H_i)=(X'/T,B_{X'},H')$
belongs to $\fM$, that is,
$(X_i,B_{X_i},H_i)\in \fM$.
In other words, each $\sD_i$ gives a (surjective) flop
over $\fM$.
In general, we say that $\sD$ is {\em flopping\/} over $\fM$
if $(X',B_{X'},H')\in \fM$.
This property is compartible with algebraic equivalence over $T$.
For $g$ given by $\sD$,
the induced map of b-sheaves transforms the
polarization $\sH'$ into the b-polarization
$\sD=c^*\sH'$ over $T$.
The same holds for generic $\sD\in \bNS X_\eta$.
We suppose that all $\sD_i\in \Lambda$ and
$\Lambda$ is also invariant under every $g^*$,
it is automatically under $c^*$.
The latter means that $c$ determines a unique lattice
$\sH'\in\Lambda'=c_*\Lambda$.
On each rather general special fiber $X_t$,
there exists a natural lattice structure
$$
\Lambda\hookrightarrow \bNS X_t, \sD\mapsto \sD_t=\sD\rest{X_t}
$$
with the image $\Lambda=\Lambda_t\subseteq\bNS X_t$.
(This is actually injection for
connected families.)
By construction $\sH_t\in \Lambda_t$ and
$(X/T,B,H)$ is a family of triples $(X_t,B_t,H_t\in\Lambda_t)$.
The flop $c$ in $\sD\in \Lambda$
is a flop of such families that can be extended by $g\rest{X'}$ into
families of
the moduli of triples with a lattice structure.
Under canonical isomorphism of lattices:
$c^*\sH'=\sD$.

Step 3. We can convert
a flop $c\colon (X/T,B,H)\dashrightarrow (X'/T,B_{X'},H')$
into an autoflop over $T$,
if there exists an isomorphism
$g\rest{X'}\colon(X'/T,B_{X'},H')\to (X/T,B,H)$.
Then the composition gives a flop
$g=g\rest{X'}c\colon (X/T,B,H)\dashrightarrow (X/T,B,H)$.
This flop is fiberwise in the following sense:
there exists an isomorphism $g_T\colon T\to T$
such that, for each $t\in T$,
$g(X_t)=X_{g_T t}$ and $g$ induces the log flop
$$
g\rest{X_t}\colon (X_t,B_{X_t})\dashrightarrow
(X_{g_T t},B_{X_{g_T t}}).
$$

Typically this (existence) does not holds even for
a universal family $(X/T,B,H)$ of fine moduli.
If $c$ preserves a universal family then we
have an isomorphism $g\rest{X'}$ and a generic flop $g=g\rest{X'}c$.

This can be done for a fine moduli with lattice:
$(X_t,B_t,H_t\in\Lambda_t)$.
Indeed, if
$(X_t,B_t,H_t\in\Lambda_t)=(X_s,B_s,H_s\in\Lambda_s)$
is an isomorphism (unique and canonical for fine moduli).
Then the canonical identification $\Lambda_t=\Lambda=\Lambda_s$
is given, that is, the log isomorphism
$h_{t,s}\colon
(X_t,B_t,H_t\in\Lambda_t)=(X_s,B_s,H_s\in\Lambda_s)$
transforms each sheaf $\sD_s\in\Lambda_s$ into
sheaf $h_{t,s}^*\sD_s=\sD_t$ (modulo algebraic equivalence),
in particular, the polarization $\sH_t=h_{t,s}^*\sH_s$.
Thus isomorphic triples go under $\sD$-flop
into isomorphic triples and the same for $g\1$
given by $g_*\sH$.
Note also that such a flop changes the polarization:
$\sH'=\sD$ (usually $\not=\sH$) and
$H'=\sH_{X'}'$ is the polarization of
$(X'/T,B_{X'},H'\in \Lambda'),c^*\Lambda'=\Lambda$.
Thus a universal family will be preserved
for triples with the lattice structure.

Step 4.
Any moduli of lattice triples $(X,B,H\in \Lambda)$
can be converted into fine moduli adding
an extra rigidity structure $R$.
It is sufficient that
$\Aut (X,B,H\in \Lambda)=\{\id_X\}$.
The group $\Aut (X,B,H\in \Lambda)$ is
tame by Theorem~\ref{aut_tame}.
E.g., if $R\in X$ will be
a rather general $l$-taple of points in $X$.
Then $\Aut (X,B,H\in \Lambda,R)=\{\id_X\}$ and
generically the moduli $\fM$ of such quadruples
are fine.
The family $(X/T,B,H\in\Lambda)$ of Step 3
can be converted into a family of quadruples
$(X/T,B,H\in\Lambda,R)$ which is dominant on $\fM$.
This can be done by an appropriate base change
and taking an open subfamily after that.
E.g., for the moduli with a $l$-taple $R$,
take a base change under the fiber power $f^l\colon X_T^l\to T$
over $T$ and
then an open subset in $X_T^l$ corresponding
to $l$-taples $R\in X_T^l$ with
$\Aut (X_t,B_{X_t},H_t\in
\Lambda_t,R)=\{\id_{X_t}\},t=f^l(R)$.

Now we can suppose that each flop given by $\sD_i$ and
all other flopping divisors $\sD$ can be extended to
a generic flop of $(X_\eta\to\eta/k,B_{X_\eta})$.
In general, the group of generic flops
can be infinite.

Step 5.
The required finiteness of $\GM\bnd$
follows from the finiteness of the quotient
group of Corollary~\ref{it_generic_Kawamata} and
thus by it.
Indeed, the objects of $\GM\bnd$ are
given by isomorphism classes of pairs $(X_t,B_{X_t})$ in
the orbit of sufficiently general fiber
$(X_t,B_{X_t},H_{X_t})$ under
the action of $\Bir(X_\eta\to\eta/k,B_{X_\eta})$.
(A chain of bounded flops.)
Such a fiber exists by the dense property of Step 1 and,
if $\GM\bnd$ is infinite up to isomorphism,
then the orbit is well-defined for
most (a dense subset) of such objects.
(Actually it is possible to make for all points
using klt limits of $0$-pairs.)

By definition of $\aBir(X_\eta\to\eta/k,B_{X_\eta})$
the orbit for this subgroup of almost autoflops
is isotrivial: a single pair $(X_t,B_{X_t})$
up to isomorphism for rather general $t$.
Thus by the corollary the set of pairs
in the orbit up to log isomorphism for
the whole group is finite too.

Finally, a bounded set of flops of a finite set
of pairs is finite up to log isomorphisms.
\end{proof}

Department of Mathematics

Johns Hopkins University

Baltimore, MD 21218, USA

shokurov@math.jhu.edu

Mathematical Institute

Russian Academy of Sciences

Moscow, Russia

shokurov@mi.ras.ru


\begin{thebibliography}{99}

 \bibitem[Am]{Am} F.~Ambro, {\em Fibering log varieties\/},
preprint, July 27, 2003.

 \bibitem[BCHM]{BCHM} C.~Birkar, P.~Cascini,
C.~Hacon, and J.~McKernan, {\em Existence of minimal models
for varieties of log general type\/},
J. Amer. Math. Soc., Vol. 23 (2010) No. 2, 405--468.

 \bibitem[FC]{FC} O.~Fujino, and Y.~Gongyo, {\em Log
pluricanonical representations and abundace conjecture\/},
arXive:1104.0361v3

 \bibitem[ISh]{ISh} V.A.~Iskovskikh, and V.V.~Shokurov,
{\em Birational models and flips\/}, Russian Mathematical Surveys,
{\bf 60} (2005), 27--94.

 \bibitem[Kol11]{Kol11} J.~Koll\'ar, {\em Sources of log canonical centers\/},
arXive:1107.2863v3

 \bibitem[NU]{NU} I.~Nakamura, and K.~Ueno, {\em An addition
formula for Kodaira dimension of analytic fiber bundles whose fiber are
Moishezon manifolds\/}, J. Math. Soc. Japan, {\bf 25} (1973), 363--371.

 \bibitem[PShSh]{PShSh} I.I.~Pjateckii-Shapiro, and
I.R.~Shafarevich, {\em A Torelli theorem for algebraic surfaces of
type K3\/}, Izv. Akad. Nauk SSSR Ser. Math.
{\bf 35} (1971), 530--572.

 \bibitem[PSh]{PSh} Yu.G.~Prokhorov, and V.V.~Shokurov,
{\em Toward the second main theorem on complements\/},
J. Algebraic Geometry {\bf 18} (2009), 151--199.

 \bibitem[S]{S} F.~Sakai, {\em Kodaira dimensions of
complements of divisors\/}, Complex Analysis and Algebraic
Geometry, Iwanami Shoten, Tokyo, 1977, 239--257.

 \bibitem[ShC]{ShC} V.V.~Shokurov, and S.R.~Choi,
{\em Geography of log models: theory and applications\/},
Cent. Eur. J. Math. {\bf 9} (2011), 489--534.

 \bibitem[U]{U} K.~Ueno, {\em Classification theory of algebraic
varieties and compact complex spaces\/}, Lecture Notes in Math.
{\bf 439}, Springer, Berlin, 1975.
\end{thebibliography}
\end{document}